\documentclass[a4paper, 12pt]{article}

\usepackage[a4paper]{geometry}
\usepackage{amsmath, amsfonts, amssymb, amsthm, mathrsfs, mathtools}
\usepackage[footnotehyper]{nicematrix}
\usepackage{csquotes}
\usepackage{indentfirst}
\usepackage{chngcntr}
\usepackage{enumitem}
\usepackage{titlesec}
\usepackage{caption}
\usepackage{cmap}
\usepackage[T1]{fontenc}
\usepackage{floatrow}
\usepackage{algpseudocode}
\usepackage{algorithm}
\usepackage{lmodern}
\usepackage[colorlinks, linkcolor = purple, citecolor = purple]{hyperref}
\usepackage{orcidlink}

\mathtoolsset{showonlyrefs=true}
\titleformat{\section}
 {\normalfont\fontsize{16}{16}\bfseries}{\thesection}{1em}{}

\DeclareFontEncoding{LS2}{}{\noaccents@}
\DeclareFontSubstitution{LS2}{stix}{m}{n}
\DeclareSymbolFont{arrows3}{LS2}{stixtt}{m}{n}
\DeclareMathSymbol{\squareulblack}{\mathord}{arrows3}{"88}

\newcommand*{\eqdef}{\mathop{\overset{\mathrm{def}}{\resizebox{\widthof{\ensuremath{\mathop{\overset{\mathrm{def}}{=}}}}}{\heightof{=}}{=}}}}

\theoremstyle{definition}

\newtheorem{proposition}{Proposition}
\newtheorem{lemma}{Lemma}
\newtheorem{theorem}{Theorem}
\newtheorem{problem}{Problem}
\newtheorem{remark}{Remark}
\newtheorem{corollary}{Corollary}[theorem]
\newtheorem*{definition*}{Definition}
\newtheorem*{lemma*}{Lemma}
\newtheorem*{theorem*}{Theorem}
\newtheorem*{corollary*}{Corollary}

\DeclareMathOperator*{\tr}{tr}
\DeclareMathOperator*{\rank}{rank}
\DeclareMathOperator*{\diag}{diag}
\DeclareMathOperator*{\adj}{adj}
\DeclareMathOperator*{\im}{Im}
\DeclareMathOperator*{\argmin}{argmin}

\renewcommand{\geq}{\geqslant}
\renewcommand{\leq}{\leqslant}

\newcommand{\R}{\mathbb{R}}
\newcommand{\bP}{\mathbb{P}}
\newcommand{\E}{\mathbb{E}}
\newcommand{\vol}{\mathrm{Vol}}

\newcommand{\cS}{\mathcal{S}}
\newcommand{\cR}{\mathcal{R}}
\newcommand{\cT}{\mathcal{T}}

\newcommand{\normxi}[1]{\Vert{#1}\Vert_\xi}
\newcommand{\normf}[1]{\Vert{#1}\Vert_F}
\newcommand{\norms}[1]{\Vert{#1}\Vert_2}
\newcommand{\matdec}[1]{\mathrm{#1}}

\newcommand{\regpinv}[2]{#1^{\dagger(#2)}}

\makeatletter
\newcommand{\MultL}[1]{%
    \parbox[t]{\dimexpr\linewidth-\ALG@tlm-1em}{\raggedright #1\strut}
}
\makeatother

\algblockdefx[Phase]{Phase}{EndPhase}[1]{\textbf{#1}}{}
\algtext*{EndPhase}

\begin{document}

\title{Subset selection for matrices by volume sampling}

\author{
  \orcidlink{0009-0008-3133-064X} Ivan Kozyrev%
  \thanks{Marchuk Institute of Numerical Mathematics of Russian Academy of Sciences, Gubkina St., Moscow, 119333, Russia}
  \thanks{Moscow Institute of Physics and Technology, Institutskiy Per., Dolgoprudny, Moscow Region, 141701, Russia}
  \and
  \orcidlink{0000-0002-9956-8223} Alexander Osinsky%
  \thanks{Skolkovo Institute of Science and Technology, Bolshoy Blvd., Moscow, 121205, Russia}
  \footnotemark[1]
}

\maketitle

\begin{abstract}
We address the Subset selection problem for matrices, where the goal is to select a subset $\mathcal{S}$ of $k$ column indices from a \enquote{short-and-fat} matrix $X \in \mathbb{R}^{m \times n}$, such that the sampled submatrix $X_{\mathcal{S}}$ has $\|X_{\mathcal{S}}^\dag X\|_F$ as small as possible. Our approach is centered on volume sampling, which attains the tightest known bound on this objective in expectation. As our primary contribution, we propose a new deterministic algorithm, Forward derandomized volume sampling (FDVS), which provably attains this bound and has asymptotic complexity $O(nkm)$. In contrast, the complexity of all previously known algorithms with this guarantee scales at least quadratically in $n$ when $k \ll n$, making FDVS particularly attractive for very wide matrices. In addition, we systematically structure the landscape of volume sampling methods: we propose a simple $O(nm^2)$ algorithm for exact forward volume sampling, show that a known deterministic method is in fact a derandomization of Reverse iterative volume sampling, and derive a modification of the latter that avoids repeated SVD downdating, along with a fast greedy forward variant. The algorithms are verified and compared in numerical experiments involving optimal experimental design, sensor placement, and DLRA-DEIM.\vspace{6pt}
    
\vspace{6pt}
    
\end{abstract}

\section{Introduction}

\subsection{Subset selection for matrices}

This paper addresses the problem of finding highly non-degenerate submatrices within a given matrix $X \in \R^{m \times n}$ (where $m < n$ and often $m \ll n$). The problem is formally defined as follows:

\begin{problem}\label{prm:subset_selection}
Given a full-rank matrix $X \in \R^{m \times n}$ with $m < n$ and a sampling parameter $k \in \overline{m,n} \eqdef \{m, m + 1, \dots, n\}$, find a set of column indices $\cS_{opt} \subseteq \overline{1, n}$ such that $|\cS_{opt}| \leq k$, $\rank(X_{\cS_{opt}}) = m$, and $\normxi{X_{\cS_{opt}}^\dag X}$ is minimized. Formally,
\[
\cS_{opt} \in \argmin_{\cS \in \mathcal{F}(X, k)} \normxi{X_\cS^\dag X}\,,
\]
where $\mathcal{F}(X, k) =\left\{\cS \subseteq \overline{1,n} \,:\, |\cS| \leq k \text{ and } \rank \left( X_\cS \right) = m \right\}$ and $\xi$ can denote either the Frobenius $(\xi = F)$ or the spectral $(\xi = 2)$ norm.
\end{problem}

In this article, we focus on the Frobenius norm version of this problem and its relation to the notion of volume and volume sampling.

The matrix $X_\cS^\dag X$ possesses an intuitive geometric interpretation. Its $j$-th column represents the minimal-length vector of coefficients required to express the $j$-th column of $X$ as a linear combination of the columns of $X_\cS$. Therefore, minimizing $\normxi{X_\cS^\dag X}$ corresponds to finding a submatrix whose columns can most stably represent all the columns of $X$.

Currently, no polynomial-time algorithm is known to find the exact optimum $\cS_{opt}$. Furthermore, in the case where $\xi = F$ and $k = m$, it has been proven that the problem is NP-hard to approximate within a fixed constant~\cite{doi:10.1287/moor.2021.1129}. This theoretically justifies the need for efficient approximation techniques.

\subsection{Applications}

Problem~\ref{prm:subset_selection} has numerous applications across statistics, machine learning, and numerical linear algebra. We broadly group these applications into two categories based on the specific objective function they seek to minimize.

The first category minimizes exactly $\normxi{X_\cS^\dag X}$. This objective arises naturally in least-squares regression~\cite{BOUTSIDIS2014273}, rank-deficient least squares~\cite{10.1145/1132973.1132981, doi:10.1137/090780882}, preconditioning~\cite{Arioli2015-eb}, low-rank approximation~\cite{Osinsky2023-bg, doi:10.1137/140977898, GOREINOV19971}, DEIM~\cite{deim}, DLRA-DEIM~\cite{carrel2025, nakao2026adaptiveconservativelowrankimex, DEKTOR2025143}, graph signal processing~\cite{chen2015discrete, tsitsvero2016signals, ZHANG2025109668}, sensor placement~\cite{sensor_placement} and the search for low-stretch spanning trees~\cite{faster_subset_selection}.

Conversely, the second category involves optimizing the classical objective $\normxi{X_\cS^\dag}$~\cite{faster_subset_selection}. This is prevalent in optimal experimental design~\cite{Huan_Jagalur_Marzouk_2024, doi:10.1287/moor.2021.1129} and the search for subgraphs with high algebraic connectivity~\cite{Fiedler1973, lamperskisimple}. Note, however, that even in experimental design there are formulations (V-, G- and D-\hspace{0pt}optimality) that do not require the classical objective. Moreover, even when the goal is to minimize the spectral norm of the pseudoinverse (E-optimality), it is common to use algorithms that operate on the orthonormal factor $Q$~\cite{faster_subset_selection} (where $X = Y Q$, $Y$ is invertible, and $Q$ has orthonormal rows), since in this case
\[
\frac{\norms{X_\cS^\dag}}{\norms{X^\dag}} = \frac{\norms{Q_\cS^\dag Y^{-1}}}{\norms{Q^T Y^{-1}}} = \frac{\norms{Q_\cS^\dag Y^{-1}}}{\norms{Y^{-1}}} \leq \norms{Q_\cS^\dag} = \norms{X_\cS^\dag X}\,.
\]
Thus, E-optimal design can also be approached with algorithms from the first category, and any bound on $\norms{X_\cS^\dag X}$ carries over to $\norms{X_\cS^\dag} / \norms{X^\dag}$.

\subsection{Volume and volume sampling}\label{sec:intro_vol}

The volume of an $m \times k$ ($k \geq m$) submatrix $X_\cS$ is defined by 
\[
\vol^2(X_\cS) \eqdef \det (X_\cS X_\cS^T)\,.
\]
This concept is intrinsically connected to the objective of the Frobenius norm version of Problem~\ref{prm:subset_selection}. Indeed, for any full-rank $X_\cS$,
\begin{equation}\label{eqn:frob_through_vol}
\normf{X_\cS^\dag X}^2 = \frac{1}{k - m + 1}\left(\sum_{r \in \cS,\ s \notin \cS - r}\frac{\vol^2(X_{\cS -r + s})}{\vol^2(X_\cS)} - (k - m)(n - k + 1 - m)\right)\,,
\end{equation}
where we denote $\cS - j \eqdef \cS \setminus \{j\}$ and $\cS + j \eqdef \cS \cup \{j\}$. We prove~\eqref{eqn:frob_through_vol} in Appendix~\ref{app:frob_through_vol}.

From~\eqref{eqn:frob_through_vol}, one can observe that a submatrix whose volume is larger than that of all its neighbors $X_{\cS -r + s}$ satisfies~\cite{Osinsky23, osinsky2026subsetselectionmatricescolumn}
\begin{equation}\label{eqn:maxvol_bound}
    \normf{X_\cS^\dag X}^2 \leq m \frac{n - m + 1}{k - m + 1}\,.
\end{equation}
Alternatively, one can sample submatrices with probabilities proportional to their squared volume. This is called volume sampling, and it guarantees~\cite{derezinski2018reverse}
\[
\E\normf{X_\cS^\dag X}^2 = m \frac{n - m + 1}{k - m + 1}\,.
\]
As we show in the present article, volume sampling can be derandomized, again providing~\eqref{eqn:maxvol_bound}. Note that~\eqref{eqn:maxvol_bound} is the tightest known achievable upper bound on $\normf{X_\cS^\dag X}^2$. To our knowledge, all algorithms that reach this bound (even approximately) are based on one of these two ideas. 

\subsection{Our contributions}

Our primary contribution is FDVS (Forward derandomized volume sampling, Algorithm~\ref{alg:fdvs}) for the Frobenius norm version of Problem~\ref{prm:subset_selection}. It deterministically reaches bound~\eqref{eqn:maxvol_bound} in $O(nkm)$ operations. Previously, only Algorithm~1 in~\cite{faster_subset_selection} and Algorithm~2 in~\cite{li2017polynomial} provably reached~\eqref{eqn:maxvol_bound} in polynomial time, and their complexity is at least quadratic in $n$ in the common setting $k \ll n$. This makes FDVS an attractive choice for extremely large $n$.

Additionally, we systematically structure the landscape of volume sampling methods (see Table~\ref{tbl:alg_summary}):
\begin{itemize}
    \item We describe FIVS (Forward iterative volume sampling, Algorithm~\ref{alg:forward_vol_sample}), a simple algorithm that performs exact volume sampling in $O(nm^2)$ operations by selecting columns one at a time.
    \item We show that Algorithm~2 in~\cite{faster_subset_selection} is a derandomization of RIVS (Reverse iterative volume sampling, Algorithm~3 in~\cite{derezinski2018reverse}), and give an implementation of it (RDVS, Algorithm~\ref{alg:lq_frob_rem}) that avoids $\matdec{SVD}$ downdating while keeping the same complexity.
    \item Frobenius selection (Algorithm~\ref{alg:frob_sel}) turns the column removal procedure of RDVS into a fast $O(nkm)$ selection procedure with a very small constant.
\end{itemize}

Together with RIVS~\cite{derezinski2018reverse}, these algorithms complete the picture of volume sampling, covering all four regimes: forward and reverse, randomized and derandomized.

Lastly, we compare the algorithms numerically on optimal experimental design, sensor placement, and DLRA-DEIM (Section~\ref{sec:experiments}). The newly developed Frobenius selection is the fastest way to obtain a submatrix with near-best $\normf{X_\cS^\dag X}$, while FDVS provides the best available guarantees at the same asymptotic cost. Section~\ref{sec:conclusion} summarizes the resulting practical guidelines for choosing an algorithm.

\subsection{Related work}

The development of fast algorithms for Problem~\ref{prm:subset_selection} has been an active research area in recent years. For a comprehensive survey, see~\cite{faster_subset_selection}. Newer algorithms designed primarily for the spectral norm version of Problem~\ref{prm:subset_selection} include methods based on the barrier potential~\cite{subset_spectral} and interlacing polynomials~\cite{subset_selection_with_fixed_blocks, xu2025newlowerboundsminimum}. A study of various heuristics with good empirical performance is presented in~\cite{doi:10.1137/19M1307391}. Of separate interest are the Lebesgue-type bounds for Problem~\ref{prm:subset_selection} obtained in~\cite{allen2021near, pmlr-v99-madan19a}.

Fast algorithms specifically designed for the Frobenius norm version of Problem~\ref{prm:subset_selection}, along with our new algorithms, are listed in Table~\ref{tbl:alg_summary}.

\begin{table}[htbp]
\renewcommand{\arraystretch}{1.15}
\resizebox{\linewidth}{!}{%
\begin{NiceTabular}{|l|c|c|}
\hline
Method & Bound on $\normf{X_\cS^\dag X}^2$ & Complexity \\
\hline
\Block{1-3}{\textit{Existing algorithms}}\\
\hline

\Block{2-1}{Algorithm~2 in~\cite{faster_subset_selection}} & \Block{2-1}{$m\frac{n - m + 1}{k - m + 1}$} & \Block{2-1}{$O(nm^2 + nm(n - k))$\footref{fnt:svd_downdating}\footnotetext[1]{\label{fnt:svd_downdating}The stated complexities of FDVS and Algorithm~2 in~\cite{faster_subset_selection} rely on fast multiplication by a Cauchy-like matrix of eigenvectors, see the proof of Theorem~\ref{thm:vol_derand}; in Algorithm~2 in~\cite{faster_subset_selection} it arises through $\matdec{SVD}$ downdating~\cite{doi:10.1137/S0895479893251472}. Its hidden constant is polylogarithmic in the target precision, and structured multiplication outperforms the direct one only at $m \sim 10^3$~\cite{doi:10.1137/130945995}. Our implementation uses direct multiplication, which gives $O(nkm^2)$ for FDVS in practice.}} \\
& & \\
\hline

\Block{2-1}{RIVS~\cite{derezinski2018reverse}} & \Block{2-1}{$m\frac{n - m + 1}{k - m + 1}$ in expectation}  & \Block{2-1}{$O(nm^2)$ in expectation} \\
& & \\
\hline

\Block{2-1}{Dominant~\cite{Osinsky23}, $c > 1$} & \Block{4-1}{$m\frac{n - m + 1}{k - m + 1} + \frac{(c^2 - 1)k(n - k)}{k - m + 1}$} & \Block{2-1}{$O(nkm \log_ck)$} \\
& & \\
\cline{1-1}\cline{3-3}

\Block{2-1}{Dominant-split~\cite{osinsky2026subsetselectionmatricescolumn}, $c > 1$} & & \Block{2-1}{$O(nkm + nm^2 \log\log m$ \\ $+\, nm^2 / \log c)$} \\
& & \\
\hline

\Block{1-3}{\textit{New algorithms}}\\
\hline
\Block{2-1}{FDVS (Algorithm~\ref{alg:fdvs})\footref{fnt:equivalence_to_CCA}\footnotetext[2]{\label{fnt:equivalence_to_CCA}In the case $k = m$, FDVS is equivalent to Frobenius selection, see Remark~\ref{rem:fdvs_frob_sel_equivalence}. The latter is just Algorithm~2 in~\cite{OSINSKY2025} in that case.}} & \Block{4-1}{$m\frac{n - m + 1}{k - m + 1}$} & \Block{2-1}{$O(nkm)$\footref{fnt:svd_downdating}} \\
& & \\
\cline{1-1}\cline{3-3}
\Block{2-1}{RDVS (Algorithm~\ref{alg:lq_frob_rem})} & & \Block{2-1}{$O(nm^2 + nm(n - k))$\footref{fnt:no_svd_downdating}\footnotetext[3]{\label{fnt:no_svd_downdating}In contrast to Algorithm~2 in~\cite{faster_subset_selection}, RDVS avoids $\matdec{SVD}$ downdating and the corresponding large hidden constant.}} \\
& & \\
\hline
\Block{2-1}{FIVS (Algorithm~\ref{alg:forward_vol_sample})} & \Block{2-1}{$m\frac{n - m + 1}{k - m + 1}$ in expectation} & \Block{2-1}{$O(nm^2)$\footref{fnt:fast_vs}\footnotetext[4]{\label{fnt:fast_vs} If the input matrix has orthonormal rows, or if its orthonormal factor is precomputed, the complexity can be reduced to $O(nm + m^3\log m)$ in expectation via rejection sampling~\cite{pmlr-v99-derezinski19b, barthelme2023fastersamplerdiscretedeterminantal}.}} \\
& & \\
\hline
\Block{2-1}{Frobenius selection (Algorithm~\ref{alg:frob_sel})\footref{fnt:equivalence_to_CCA}} & \Block{2-1}{$\frac{m^2}{k}(n - m + 1)$} & \Block{2-1}{$O(nkm)$} \\
& & \\
\hline
\end{NiceTabular}%
}
\caption{Summary of modern approximation algorithms for the Frobenius norm version of Problem~\ref{prm:subset_selection}. All algorithms accept arbitrary $k \in \overline{m, n}$.}
\label{tbl:alg_summary}
\end{table}

We omit well-known algorithms such as Maxvol~\cite{maxvol} or sRRQR~\cite{osti_276532} from Table~\ref{tbl:alg_summary}. They only accept $k = m$ and are equivalent to the Dominant~\cite{Osinsky23} algorithm restricted to that case. Another related algorithm is Householder-based maxvol2~\cite{Mikhalev2015, osinsky2019rectangularmaximumvolumeprojective}, which greedily selects columns maximizing the volume of the submatrix. This approach, however, does not provide any known guarantees except when columns may be selected with repetition~\cite{pmlr-v99-madan19a}.

\subsubsection{Volume sampling algorithms}

We discuss algorithms related to volume sampling separately, as they are of special interest in the present article.

Polynomial time algorithms for forward volume sampling and its derandomization in the $k > m$ setting were first proposed in~\cite{li2017polynomial}. There, the authors use techniques similar to ours, but do not optimize the algorithms, resulting in complexities of $O(n^4 k)$ for exact volume sampling, $\tilde{O}(n k^3 m^2)$ for approximate volume sampling, and $O(n^4 k m)$ for the derandomized version.

The complexity of volume sampling was subsequently reduced by Dereziński and Warmuth~\cite{derezinski2018reverse}, who proposed RIVS. Additionally, Dereziński, Warmuth and Hsu~\cite{derezinski2022unbiased} decomposed volume sampling for $k > m$ into volume sampling of $m$ columns followed by uniform sampling of the remaining $k - m$, and proposed distortion-free intermediate sampling, a technique that accelerates volume sampling when $n \gg m^2$. Instead of sampling from all $n$ columns, it draws an intermediate pool of $O(m^2)$ columns according to approximate leverage scores, and accepts the whole pool with a probability that corrects the approximation exactly. An arbitrary exact sampler (e.g.\ RIVS or FIVS) can then be applied to the pool, so that its cost no longer depends on $n$. The resulting complexity of obtaining the pool is $O(nm\log n + m^4\log m)$ with probability $1 - \delta$, the hidden constant being logarithmic in $\delta^{-1}$.

\subsection{Organization}

The remainder of the paper is organized as follows. Section~\ref{sec:fdvs} presents our main contribution, the FDVS algorithm, together with the theory connecting Problem~\ref{prm:subset_selection} to the column subset selection problem. Section~\ref{sec:fivs} introduces the simple exact volume sampling method FIVS. Section~\ref{sec:frob_improve} revisits Algorithms~1 and~2 in~\cite{faster_subset_selection}, interprets them as derandomized versions of volume sampling, and removes the hidden logarithmic factor from their running time. Section~\ref{sec:frob_sel} develops the fast Frobenius selection algorithm and discusses the limitations of greedy selection. Finally, Section~\ref{sec:experiments} reports numerical experiments in the settings of optimal experimental design, sensor placement, and DLRA-DEIM. Sections~\ref{sec:fivs}--\ref{sec:frob_sel} are largely independent and may be read in any order.

\section{Forward derandomized volume sampling algorithm}\label{sec:fdvs}

The core idea of the algorithm presented in this section is closely related to the column subset selection problem (CSS/CSSP), where the goal is to select $k$ columns from an $M \times N$ matrix $A$ in order to minimize the Frobenius norm error of the column approximation $\normf{A - \bP _{A_\cS}A}$. Here, $\cS$ is a set of selected column indices of cardinality $k$, and $\bP_{A_\cS} = A_\cS A_\cS^\dag$ is the orthogonal projector onto the span of the columns of $A$ indexed by $\cS$. Note that in this problem $k \leq \rank A \leq \min\{M, N\}$.

In~\cite{Osinsky2023-bg}, Osinsky highlighted the connection between CSSP and Problem~\ref{prm:subset_selection}. In particular, he made the following remark, which we replicate here for convenience:

\begin{remark}[Remark~3 in~\cite{Osinsky2023-bg}]
If one has an algorithm that constructs a column matrix approximation, then applying it to 
\begin{equation}\label{eqn:A_def}
    A = 
    \begin{pNiceMatrix}[hvlines, margin, cell-space-limits = 3pt]
    \Block{1-1}{Q} \\
    \Block{1-1}{\varepsilon I_{n \times n}} \\
    \end{pNiceMatrix}\,,\ \text{where}\ \varepsilon > 0\ \text{and}\ Q \in \R^{m \times n}\ (m < n\,,\ QQ^T = I)\,,
\end{equation}
provides a method that finds $Q_\cS \in \R^{m \times k}$, such that
\begin{equation}\label{eqn:ssfm_through_cssp}
    1 + \frac{\normf{Q_\cS^\dag}^2 - k}{n - m} \leq \frac{\normf{A - \bP_{A_\cS} A}^2}{\varepsilon^2 (n - m)} + O(\varepsilon^2)\,.
\end{equation}
Therefore, algorithms that construct column approximations also allow finding highly non-degenerate submatrices.
\end{remark}

However, to make this approach competitive, we first need to answer several questions:
\begin{enumerate}
    \item Which column approximation algorithm provides the tightest bounds on the pseudoinverse of the selected submatrix?
    \item Problem~\ref{prm:subset_selection} requires the selected submatrix to have full rank. Will this condition be satisfied for the proposed method?
    \item Directly constructing matrix $A$~\eqref{eqn:A_def} and running the CSSP algorithm on it is impractical. How can we utilize the structure of $A$ to make the algorithm computationally efficient?
\end{enumerate}

We now answer the first question. A good choice here is to use an algorithm based on the derandomization of volume sampling, e.g., Algorithm~4 in~\cite{5671202} or Algorithm~1 in~\cite{doi:10.1137/19M1281848} (the latter is a more stable version of the former). These algorithms not only provide bounds of the form $\normf{A - \bP_{A_\cS} A}^2 \leq (k + 1)\normf{A - A_k}^2$, where $A_k$ is the truncated $\matdec{SVD}$ of $A$ of rank $k$, i.e., its best rank-$k$ approximation in the Frobenius norm, but also allow for finer guarantees based on the singular spectrum of the matrix $A$.

The key idea in these algorithms is to use the conditional mathematical expectation
\begin{equation}\label{eqn:E_A_def}
\E_\cS(A) \eqdef \E\left(\normf{A - \bP_{A_\cT} A}^2\ \vert\ \cS \subseteq \cT \subseteq \overline{1, n}\,,\ \vert \cT \vert = k\right),
\end{equation}
where the subsets of indices $\cT$, $\vert \cT \vert = k$, are sampled proportionally to the squared volume of a matrix $A_\cT$, $\det(A_\cT^T A_\cT)$. Lemma~21 in~\cite{5671202} yields the explicit formula for this conditional expectation:
\begin{equation}\label{eqn:E_S_A}
\E_\cS(A) = (k + 1)\left\vert\frac{c_{n - k + t - 1}\left(BB^T\right)}{c_{n - k + t}\left(BB^T\right)}\right\vert = (k + 1)\left\vert\frac{c_{n - k + t - 1}\left(B^TB\right)}{c_{n - k + t}\left(B^TB\right)}\right\vert\,,
\end{equation}
where $B = A - \bP_{A_\cS}A$, $t = \vert \cS \vert \leq k$, and $c_i \equiv c_i(B^TB)$ denotes the $i$-th coefficient of the characteristic polynomial:
\[
\det\left(\lambda I - B^TB\right) = c_0 + c_1\lambda + \dots + c_n\lambda^n\,.
\]

Algorithm~4 in~\cite{5671202} and Algorithm~1 in~\cite{doi:10.1137/19M1281848} build the set $\cS$ iteratively, at each step selecting an index $j \notin \cS$ that minimizes $\E_{\cS + j}(A)$. In the end, the obtained set $\cS$ of cardinality $k$ satisfies
\begin{equation}\label{eqn:cssp_bound_c}
\normf{A - \bP_{A_\cS} A}^2 \leq \E_\varnothing(A) = (k + 1) \left\vert \frac{c_{n - k - 1}\left(A^TA\right)}{c_{n - k}\left(A^TA\right)} \right\vert\,.
\end{equation}

In the following proposition, we derive the explicit bound for $\normf{Q_\cS^\dag}^2$ in~\eqref{eqn:ssfm_through_cssp}, provided we use a CSS algorithm that provides a bound almost as tight as~\eqref{eqn:cssp_bound_c}.

\begin{proposition}\label{stt:frob_error_exp}
    Let $A$ be a matrix of the form~\eqref{eqn:A_def}. Suppose some column approximation algorithm guarantees $\normf{A - \bP_{A_\cS} A}^2 \leq \E_\varnothing(A) + O(\varepsilon^4)$ for $\vert \cS \vert = k$, where $m \leq k \leq n$. Then
    \[
    \normf{Q_\cS^\dag}^2 \leq m\frac{n - m + 1}{k - m + 1} + O(\varepsilon^2)\,.
    \]
\end{proposition}

\begin{proof}
Let us transform~\eqref{eqn:cssp_bound_c} to compute $\E_\varnothing(A)$. The coefficients of the characteristic polynomial of $A^TA$ can be expressed using elementary symmetric polynomials of the singular values of $A$, $\sigma_1(A)\geq\dots\geq\sigma_n(A)$:
\[
c_i\left(A^TA\right) = (-1)^{n - i}e_{n - i}(\sigma) = (-1)^{n - i}\sum_{1 \leq j_1 < \dots < j_{n - i} \leq n} \sigma_{j_1}^2(A)\cdot\ldots\cdot\sigma_{j_{n - i}}^2(A)\,.
\]
Substituting this into~\eqref{eqn:cssp_bound_c}, we obtain
\begin{equation}\label{eqn:cssp_bound}
\E_\varnothing(A) = (k + 1)\frac{\sum_{i_1 < \dots < i_{k + 1}}\sigma_{i_1}^2(A)\cdot\ldots\cdot\sigma_{i_{k + 1}}^2(A)}{\sum_{i_1 < \dots < i_{k}}\sigma_{i_1}^2(A)\cdot\ldots\cdot\sigma_{i_{k}}^2(A)}\,.
\end{equation}

Note that $A^TA = \varepsilon^2 I + Q^T Q$, and thus the singular spectrum of $A$ consists of precisely $m$ singular values equal to $\sqrt{1 + \varepsilon^2}$ and $n - m$ singular values equal to $\varepsilon$. Thus, in the right-hand side of~\eqref{eqn:cssp_bound}, the sums in the numerator and denominator are dominated by the terms that contain all $m$ \enquote{large} singular values of $A$:
\begin{align}\label{eqn:E_nothing}
\E_\varnothing(A) &= (k + 1)\frac{(1 + \varepsilon^2)^m\sum_{i_1 < \dots < i_{k - m + 1}}\varepsilon^{2(k - m + 1)} + O\left(\varepsilon^{2(k - m + 2)}\right)}{(1 + \varepsilon^2)^m\sum_{i_1 < \dots < i_{k - m}}\varepsilon^{2(k - m)} + O\left(\varepsilon^{2(k - m + 1)}\right)}\\
&=\varepsilon^2 (k + 1)\frac{\binom{n - m}{k - m + 1}}{\binom{n - m}{k - m}} + O(\varepsilon^4) = \varepsilon^2\frac{(k + 1)(n - k)}{k - m + 1} + O(\varepsilon^4)\,.
\end{align}

Remembering that $\normf{A - \bP_{A_\cS} A}^2 \leq \E_\varnothing(A) + O(\varepsilon^4)$ and substituting it into~\eqref{eqn:ssfm_through_cssp}, we obtain
\[
\normf{Q_\cS^\dag}^2 \leq \frac{(k + 1)(n - k)}{k - m + 1} - (n - k) + m + O(\varepsilon^2) = m\frac{n - m + 1}{k - m + 1} + O(\varepsilon^2)\,.\qedhere
\]
\end{proof}

It is easy to show that the bound $\normf{A - \bP_{A_\cS} A}^2 \leq \E_\varnothing(A) + O(\varepsilon^4) = O(\varepsilon^2)$ also guarantees that $Q_\cS$ has full rank for sufficiently small $\varepsilon$.

\begin{proposition}\label{stt:full_rank}
    Let $A$ be a matrix of the form~\eqref{eqn:A_def}, and suppose that $Q_\cS \in \R^{m \times k}$ with $k \geq m$ is rank-deficient. Then $\normf{A - \bP_{A_\cS} A}^2 \geqslant {\rm const} > 0$ as $\varepsilon \to 0$.
\end{proposition}

\begin{proof}
    Without loss of generality, assume that $\cS = \overline{1, k}$, and let $\cR = \overline{1, n} \setminus \cS$ denote the complementary set of indices. We begin by computing the explicit form of the matrix $B = A - \bP_{A_\cS}A$. To do this, we compute $A_\cS^\dag = (A_\cS^T A_\cS)^{-1}A_\cS^T$ and the resulting projection $\bP_{A_\cS}A$ using the block form of $A$. As the detailed derivation is already established in the proof of Theorem~3 in~\cite{Osinsky2023-bg}, we omit it here.

    For $B$, we obtain
    \begin{equation}\label{eqn:B_expr}
B = \left( I - \bP_{A_\cS}\right)A = 
\begin{pNiceMatrix}[vlines, margin, cell-space-limits=3pt]
\Block{2-1}{0_{m \times k}} 
    & \Block{2-1}{\left(I - Q_\cS \regpinv{Q_\cS}{\varepsilon}\right) Q_\cR} \\
    & \\ \hline
\Block{2-1}{0_{k \times k}} 
    & \Block{2-1}{- \varepsilon\regpinv{Q_\cS}{\varepsilon} Q_\cR} \\
    & \\ \hline
\Block{2-1}{0_{(n - k) \times k}} 
    & \Block{2-1}{\varepsilon I_{(n - k) \times (n - k)}} \\
    &  
\end{pNiceMatrix}\,,
\end{equation}
    where $\regpinv{Q_\cS}{\varepsilon} = ( Q_\cS^T Q_\cS + \varepsilon^2 I )^{-1} Q_\cS^T$ is the Tikhonov-regularized pseudoinverse.

    Now, consider the block $(I - Q_\cS \regpinv{Q_\cS}{\varepsilon}) Q_\cR$. We analyze it by computing the series expansion of $Q_\cS \regpinv{Q_\cS}{\varepsilon}$:
    \begin{align}\label{eqn:Q_Q_dag_series}
        Q_\cS\regpinv{Q_\cS}{\varepsilon} &= Q_\cS\left(Q_\cS^TQ_\cS + \varepsilon^2 I\right)^{-1}Q_\cS^T\\
        &= Q_\cS\left(Q_\cS^TQ_\cS + \varepsilon^2 I\right)^{-1}\left(Q_\cS^TQ_\cS + \varepsilon^2 I\right) Q_\cS^\dag - \varepsilon^2Q_\cS\left(Q_\cS^TQ_\cS + \varepsilon^2 I\right)^{-1}Q_\cS^\dag\\
        &= Q_\cS Q_\cS^\dag - \varepsilon^2 \left(Q_\cS^\dag\right)^T Q_\cS^\dag + O\left(\varepsilon^4\right)\,.
    \end{align}
    Thus, we have $(I - Q_\cS \regpinv{Q_\cS}{\varepsilon}) Q_\cR = (I - Q_\cS Q_\cS^\dag) Q_\cR + O(\varepsilon^2)$. In the case $\rank Q_\cS < m$, $(I - Q_\cS Q_\cS^\dag) Q_\cR = (I -\bP_{Q_\cS}) Q_\cR \neq 0$, as $Q_\cR$ must have at least one column not in the column span of $Q_\cS$. From this we obtain $\normf{A - \bP_{A_\cS} A}^2 \geqslant \normf{B}^2 > 0$ as $\varepsilon \to 0$.
\end{proof}

Now we proceed to the last question, which is about making the proposed method computationally efficient. To this end, we express $\E_\cS(A)$ up to $O(\varepsilon^4)$ in terms of $Q_\cS$.

\begin{lemma}\label{lem:vol_derand_lim}
    Let $A$ be a matrix of the form~\eqref{eqn:A_def}. Take $k \in \overline{m, n}$ and $\cS \subseteq \overline{1, n}$ as a subset of column indices with cardinality $t \leq k$. Let $r = \rank Q_\cS$ satisfy $m - r \leq k - t$.
    
    Then, for $B = (I - \bP_{A_\cS}) A$, the coefficients of the characteristic polynomial of $B^TB$ satisfy
    \begin{equation}\label{eqn:c_limit}
        \left\vert\frac{c_{n - k + t - 1}\left(B^TB\right)}{c_{n - k + t}\left(B^TB\right)}\right\vert = \varepsilon^2\left\vert\frac{c_{k - t + r - m + 1}(p_\cS)}{c_{k - t + r - m}(p_\cS)}\right\vert + O(\varepsilon^4)\,,
    \end{equation}
    where $c_j(p)$ denotes the $j$-th coefficient of the polynomial $p_{\cS}(x)$ given by
    \begin{equation}\label{eqn:p_definition}
        p_{\cS}(x) = (x - 1)^{n - m - t}\prod_{i = 1}^{r}\left(x - \sigma_i^2(Q_\cS)\right)\,.
    \end{equation}
    Note that while $n - m - t$ can be negative, $p_\cS(x)$ is always a polynomial. Its degree is $d = n - t + r - m$, so $k - t + r - m = d - n + k$; Algorithm~\ref{alg:fdvs} uses the latter form, where $d$ is maintained explicitly.
\end{lemma}

\begin{proof}
Without loss of generality, assume that $\cS = \overline{1, t}$, and let $\cR = \overline{1, n} \setminus \cS$ denote the complementary set of indices. We proceed in four steps:

\begin{enumerate}[wide]
\item First, we analyze the structure of $B^TB$. Using the expression for $B = A - \bP_{A_\cS}A$~\eqref{eqn:B_expr} we obtain
\[
B^TB = \begin{pNiceMatrix}[vlines, margin, cell-space-limits = 3pt]
       \Block{2-1}{0_{t \times t}} & \Block{2-1}{0_{t \times (n - t)}} \\
       & \\ \hline
       \Block{2-1}{0_{(n - t) \times t}} & \Block{2-1}{C}  \\
       &
       \end{pNiceMatrix}\,,
\]
where
\begin{align}\label{eqn:matrix_C}
    C &= Q_\cR^T\left(I - \left(\regpinv{Q_\cS}{\varepsilon}\right)^TQ_\cS^T\right)\left(I - Q_\cS \regpinv{Q_\cS}{\varepsilon}\right)Q_\cR + \varepsilon^2 Q_\cR^T \left(\regpinv{Q_\cS}{\varepsilon}\right)^T\regpinv{Q_\cS}{\varepsilon}Q_\cR + \varepsilon^2 I \nonumber \\
    &= Q_\cR^T \left[ I - Q_\cS \regpinv{Q_\cS}{\varepsilon} - \left(\regpinv{Q_\cS}{\varepsilon}\right)^TQ_\cS^T + \left(\regpinv{Q_\cS}{\varepsilon}\right)^T \left(Q_\cS^T Q_\cS + \varepsilon^2 I \right) \regpinv{Q_\cS}{\varepsilon}\right] Q_\cR + \varepsilon^2 I \nonumber \\
    &= Q_\cR^T \left[ I - Q_\cS \regpinv{Q_\cS}{\varepsilon}\right] Q_\cR + \varepsilon^2 I\,.
\end{align}
Decomposing it even further, we substitute the series expansion of $Q_\cS \regpinv{Q_\cS}{\varepsilon}$, obtained in~\eqref{eqn:Q_Q_dag_series}, into the above:
\begin{equation}\label{eqn:C_final}
    C = \underbrace{Q_\cR^T \left( I - \bP_{Q_\cS}\right) Q_\cR}_{M_1} + \underbrace{\varepsilon^2 Q_\cR^T \left(Q_\cS^\dag\right)^T Q_\cS^\dag Q_\cR}_{M_2} + \varepsilon^2 I + O\left(\varepsilon^4\right)\,.
\end{equation}
Since $c_{n - k + t - 1}(B^TB) = c_{n - k + t - 1}(C)$ and $c_{n - k + t}(B^TB) = c_{n - k + t}(C)$, we need to examine the eigenvalues of $C$, which we do in the next item.

\item Note that $M_1$ and $M_2$ are symmetric and have orthogonal images. This is because any vector $u_1 \in \im M_1$ can be expressed as $u_1 = Q_\cR^T v_1$ for some $v_1$ satisfying $Q_\cS^T v_1 = 0$, while any $u_2 \in \im M_2$ can be expressed as $u_2 = Q_\cR^T v_2$ for some $v_2 \in \im Q_\cS$. Their orthogonality is demonstrated as follows:
\[
\langle u_1, u_2 \rangle = v_1^T Q_\cR Q_\cR^T v_2 = v_1^T \left(I - Q_\cS Q_\cS^T\right) v_2 = v_1^T v_2 - \left(Q_\cS^T v_1\right)^T \left(Q_\cS^T v_2\right) = 0\,.
\]

For symmetric matrices with orthogonal images, the spectrum of their sum is the union of their individual spectra. We can thus analyze the matrices $M_1$ and $M_2$ independently:
\begin{enumerate}
    \item Observe that $\rank M_1 = \rank Q^T(I - \bP_{Q_\cS})Q = m - r$, and for $i \in \overline{1, m - r}$,
    \[
        \lambda_i(M_1) = \sigma_i^2\left( \left(I - \bP_{Q_\cS}\right)Q_\cR \right) = \sigma_i^2\left( \left(I - \bP_{Q_\cS}\right)Q \right) = 1\,.
    \]
    \item Using the fact that the ranks of $M_1$ and $M_2$ must add up to no more than $\min\{m, n - t\}$, we can bound $\rank M_2$ by $\min \{r, d\}$, where $d = n - t + r - m$. And for $1 \leq i \leq \min\{r, d\}$, $\lambda_i(M_2) = \varepsilon^2\sigma_i^2(Q_\cS^\dag Q_\cR)$.
\end{enumerate}

Using this and the spectral stability property of eigenvalues (applied to a perturbation of order $\varepsilon^4$), we can characterize the spectrum of $C$:
\begin{enumerate}
    \item There are $m - r$ \enquote{large} eigenvalues $\lambda_i = 1 + O(\varepsilon^2)$, where $i \in \overline{1, m - r}$.
    
    \item There are $d$ eigenvalues of order $\varepsilon^2$. We denote these eigenvalues by $\mu_i$ for $i \in \overline{1, d}$, given by:
    \begin{equation}\label{eqn:mu_description}
        \mu_i = 
        \begin{cases} 
            \varepsilon^2 \sigma_i^2(Q_\cS^\dag Q_\cR) + \varepsilon^2 + O(\varepsilon^4) & \text{for } 1 \leq i \leq \min\{r, d\}\,, \\
            \varepsilon^2 + O\left(\varepsilon^4\right) & \text{for } \min\{r,d\} < i \leq d\,.
        \end{cases}
    \end{equation}
\end{enumerate}

\item We proceed by further analyzing the eigenvalues of order $\varepsilon^2$. Padding a matrix with zeros does not change the non-zero singular values; thus, for $1 \leq i \leq \min\{r, d\}$,
\[
\sigma_i^2\left(Q_\cS^\dag Q_\cR\right) + 1 = \sigma_i^2\left(Q_\cS^\dag \begin{pNiceMatrix}[hvlines] \Block{2-1}{0_{m\times t}} & \Block{2-1}{Q_\cR} \\ & \end{pNiceMatrix}\right) + 1 = \sigma_i^2\left(Q_\cS^\dag Q - Q_\cS^\dag Q_\cS\right) + 1 = \sigma_i^2\left(Q_\cS^\dag\right)\,.
\]

In the case $d < r$, the remaining singular values $\sigma_{d + 1}(Q_\cS^\dag), \dots, \sigma_r(Q_\cS^\dag)$ are identically $1$. Indeed, we may write the internal direct sum $\langle Q_\cS\rangle = V \oplus W$ by letting $W$ be the intersection $\langle Q_\cS \rangle \cap \langle Q_\cR \rangle$, and defining $V$ as the orthogonal complement of $W$ within $\langle Q_\cS \rangle$. Since $\dim W = \rank Q_\cS^\dag Q_\cR \leq \min\{r, d\} = d$ and $\dim \langle Q_\cS\rangle = \rank Q_\cS = r$, it follows that $\dim V \geq r - d$. Consequently, $Q_\cS$ (and $Q_\cS^\dag$) has at least $r - d$ singular values equal to $1$, as it acts identically to $Q$ when restricted to the subspace $V$.

As a result, the eigenvalues $\mu_i$ can be described as roots of 
\[
\hat{p}_{\cS}(x) = (x - 1)^{n - m - t}\prod_{i = 1}^{r}\left(x - \sigma_i^2(Q_\cS^\dag)\right)
\]
up to $O(\varepsilon^4)$. In the case $r \leq d$, it follows directly from~\eqref{eqn:mu_description}, and in the case $r > d$, the term $(x - 1)^{n - m - t} = (x - 1)^{d - r}$ cancels out $r - d$ eigenvalues of $Q_\cS^\dag$ that are equal to $1$, again resulting in~\eqref{eqn:mu_description}.

\item Finally, we return to the ratio of the coefficients of the characteristic polynomial of $C$. They are given by the elementary symmetric polynomials of its eigenvalues:
\begin{equation}\label{eqn:c_expression}
    \left\vert\frac{c_{n - k + t - 1}(C)}{c_{n - k + t}(C)}\right\vert =\frac{e_{k - t + 1}(\lambda, \mu)}{e_{k - t}(\lambda, \mu)} \overset{(a)}{=} \frac{e_{k - t + r - m + 1}(\mu)}{e_{k - t + r - m}(\mu)} + O\left(\varepsilon^4\right) \overset{(b)}{=} \frac{e_{n - k - 1}(1/\mu)}{e_{n - k}(1/\mu)} + O(\varepsilon^4)\,,
\end{equation} 
$(a)$ holds because in the expansion of the symmetric polynomials, the dominant terms are those that contain the product of all $m-r$ of the \enquote{large} eigenvalues $\lambda_i$. Any terms not containing this product are of order $\varepsilon^2$ smaller. $(b)$ holds because $e_q(\mu_1, \dots,\mu_p) = e_{p - q}(1/\mu_1, \dots, 1/\mu_p) \cdot \prod_{i = 1}^p \mu_i$.

Since $\sigma_i^2(Q_\cS) = 1/\sigma_i^2(Q_\cS^\dag)$, the roots of $p_\cS$~\eqref{eqn:p_definition} are the reciprocals $1/\mu$ of those of $\hat{p}_\cS$. Expressing its coefficients through the elementary symmetric polynomials of these roots gives exactly~\eqref{eqn:c_expression}, which completes the proof.\qedhere
\end{enumerate}
\end{proof}

Now we have all we need to formulate Algorithm~\ref{alg:fdvs}, which approximates Problem~\ref{prm:subset_selection}. The key idea of this algorithm is to work with the matrix $Q$ from the $\matdec{LQ}$ decomposition of the input matrix $X$, and then to build the set $\cS$ iteratively, at each step selecting the index $s$ that minimizes $\vert c_{k - t + r - m + 1}(p_{\cS + s}) / c_{k - t + r - m}(p_{\cS + s}) \vert$ among $s \notin \cS$, in the notation of Lemma~\ref{lem:vol_derand_lim}. This way, we can control the conditional expectation $\E_\cS(A)$ and eventually apply Proposition~\ref{stt:frob_error_exp}. The properties of the algorithm are formally stated in Theorem~\ref{thm:vol_derand}. Stability of the algorithm and details of its numerical implementation are discussed in Remark~\ref{rem:fdvs_stability}.

\begin{algorithm}[htbp]
\caption{Forward derandomized volume sampling (FDVS).}\label{alg:fdvs}
\renewcommand{\thefootnote}{\thempfootnote}
\begin{algorithmic}[1]
\Require $X \in \R^{m \times n}$ ($m \leq n$, $\rank X = m$), sampling parameter $k \in \overline{m, n}$.
\Ensure Subset of column indices $\cS \subset \overline{1, n}$ with $|\cS| = k$.

\Phase{Initialization:}
\State Compute thin $\matdec{LQ}$ decomposition of $X$, $X = LQ$. Save matrix $Q$.
\State Set $\cS \gets \varnothing, \cR \gets \overline{1, n}, \Lambda = \diag(\lambda) \gets 0_{m \times m}$, $r \gets 0$.
\EndPhase

\For{$t = 0, \dots, k-1$}
    \Phase{Notation:}
        \State\label{lin:d_l_def} Set $l \gets \vert\{i\,|\, 0 < \lambda_i < 1 \}\vert$\footref{fnt:tolerance}, $d = n - t + r - m$, $\delta = (l < d)$\footnote{We use $\delta \in \{0, 1\}$ because the desired factor $d - l - 1$ can be negative. In that case, we divide by $(x - 1)$ after the polynomial transformations, in~\eqref{eqn:p_j_computation}.}.
        \State\MultL{
        Let $A\in \R^{(r - l) \times n}$, $B \in \R^{l \times n}$, $C \in \R^{(m - r) \times n}$ be the top, middle and bottom rows of $Q$, respectively.
        }
        \State\MultL{
        Set $\Omega = \diag(\omega_1, \dots, \omega_l)= \diag(\lambda_{r - l + 1},\dots,\lambda_r)$\footref{fnt:sorting}.
        }
    \EndPhase
    
    \Phase{Auxiliary polynomial construction:}
        \State \MultL{
        Compute coefficients $c_{d - n + k - 1}, c_{d - n + k}, c_{d - n + k + 1}$ of $p(x) = (x - 1)^{d - l - \delta}\prod_{i = 1}^{l}(x - \omega_{i})$ up to a constant.
        }
        \State \MultL{
        Compute coefficients $c_{d - n + k - 1}, c_{d - n + k}$ of $g_i(x) = p(x)/(x - \omega_i)$ for $i \in \overline{1, l}$ up to the same constant.
        }
    \EndPhase
    
    \Phase{Characteristic polynomial construction:}
        \State \MultL{
            Compute the coefficients $c_{d - n + k}, c_{d - n + k + 1}$ of
            \begin{equation}\label{eqn:p_j_computation}
                p_{\cS + j}(x) = \frac{1}{(x - 1)^{1 - \delta}}\left[xp(x) - \sum_{i = 1}^{l} B_{ij}^2 x g_i(x) - \norms{c_j}^2p(x)\right]
            \end{equation}
            up to a constant. Case $\delta = 0$ is explained in Item~\ref{itm:char_poly_construction}.
        }
    \EndPhase

    \Phase{Greedy selection and update:}
        \State Set $s \in \argmin_{j \in \cR} \vert c_{d - n + k + 1}(p_{j}) / c_{d - n + k}(p_{j})\vert$.
        \State $\cS \gets \cS + s$, $\cR \gets \cR - s$. If $\norms{c_s}^2 > 0$\footref{fnt:tolerance}, set $r \gets r + 1$.
        \State Compute eigendecomposition update\footref{fnt:sorting}: $\Lambda + q_s q_s^T = U D U^T$.
        \State \label{lin:rotation} \MultL{%
        Update diagonal matrix and rotate the remaining columns:
        \[
            \Lambda \gets D, \quad Q_\cR \gets U^T Q_\cR\,.
        \]
        }\vspace{-1em}
    \EndPhase
\EndFor
\State \textbf{return} $\cS$
\footnotetext[2]{\label{fnt:tolerance}In practice, the strict inequalities are replaced by approximate ones for small tolerance~$\tau$.}
\footnotetext[3]{\label{fnt:sorting}We assume the eigenvalues $\lambda_i$ are sorted in non-increasing order.}
\end{algorithmic}
\end{algorithm}

\begin{theorem}\label{thm:vol_derand}
    There exists a deterministic algorithm (Algorithm~\ref{alg:fdvs}) that, given a full-rank matrix $X \in \R^{m \times n}$ with $m < n$ and a sampling parameter $k \in \overline{m, n}$, constructs a subset $\cS \subseteq \overline{1,n}$ of cardinality $k$. The algorithm ensures that $X_\cS$ has full rank, and
    \[
    \normf{X_\cS^\dag X}^2 \leq m\frac{n - m + 1}{k - m + 1}\,,\quad \norms{X_\cS^\dag X}^2 \leq 1 + \frac{m(n - k)}{k - m + 1}\,.
    \]
    The algorithm runs in $O(nkm)$ operations.
\end{theorem}

\begin{proof}
    \textbf{Proof of bounds.} Consider the matrix $A$ of the form~\eqref{eqn:A_def}, constructed using the matrix $Q$ from the $\matdec{LQ}$ decomposition of $X$. In this framework, Algorithm~\ref{alg:fdvs} functions simultaneously as an approximation algorithm for Problem~\ref{prm:subset_selection} on $Q$ and as a column subset selection algorithm on $A$. We proceed in three steps:

    \begin{enumerate}[wide]
        \item First, we establish that $r$ always tracks the rank of $Q_\cS$, $\lambda$ stores the eigenvalues of $Q_\cS Q_\cS^T$ in non-increasing order, and $Q_\cR$ is expressed in an orthonormal basis where the range of $Q_\cS$ spans the first $r$ coordinates. 
    
        Proceeding by induction, observe that in the update phase, $\Lambda + q_s q_s^T$ represents the Gram matrix $Q_{\cS + s} Q_{\cS + s}^T$ in the current basis. The eigendecomposition updates $\Lambda \gets D$ and $Q_\cR \gets U^T Q_\cR$ diagonalize this matrix and align the basis with its eigenvectors. Consequently, $\norms{c_s}^2$ measures the squared projection of $q_s$ onto the orthogonal complement of $\langle Q_\cS\rangle$, ensuring that $r$ increments exactly when the rank of the selection increases.

        \item Next, we verify that the ratio $\vert c_{d - n + k + 1}(p_{j}) / c_{d - n + k}(p_{j})\vert$ agrees with the target coefficient ratio from Lemma~\ref{lem:vol_derand_lim}. We analyze the cases $\norms{c_j}^2 = 0$ and $\norms{c_j}^2 > 0$ separately.

        \begin{enumerate}[wide]
            \item Let $\norms{c_j}^2 = 0$ for a given $j \in \cR$. Then canceling out $d - l$ singular values equal to $1$ in~\eqref{eqn:p_definition} produces
            \[
            p_{\cS + j}(x) = (x - 1)^{d - l - 1} \det\left(xI - \Omega - b_jb_j^T\right)\,.
            \]
            The polynomial $p_{j}$ matches the above exactly. Indeed, the matrix determinant lemma yields:
            \begin{equation}\label{eqn:p_j_transform}
            \begin{aligned}
                p_{\cS + j}(x)
                &= (x - 1)^{d - l - 1}\det\left(xI - \Omega\right) - (x - 1)^{d - l - 1} b_j^T\adj\left(xI - \Omega\right)b_j\\
                &= \frac{1}{(x - 1)^{1 - \delta}}\left[p(x) - \sum_{i = 1}^{l} B_{ij}^2 g_i(x)\right]\,.
            \end{aligned}
            \end{equation}
            Therefore, in this case $p_j(x) = x p_{\cS + j}(x)$, which absorbs the index shift by one.

            \item Now let $\norms{c_j}^2 > 0$. In that case, after canceling $d - l$ singular values equal to~1,~\eqref{eqn:p_definition} reduces to
            \begin{equation}\label{eqn:p_case_2}
            p_{\cS + j}(x) = (x - 1)^{d - l - 1}\det\left(xI - \hat{\Omega} - \hat{b}_j\hat{b}_j^T\right)\,,
            \end{equation}
            where $\hat{\Omega} = \diag(\omega_1, \dots, \omega_l, 0)$ and $\hat{b}_j = (B_{1j}, \dots, B_{lj}, \norms{c_j})^T$. Applying the matrix determinant lemma again yields
            \begin{align*}
                p_{\cS + j}(x)
                &= (x - 1)^{d - l - 1}\det\left(xI - \hat{\Omega}\right) - (x - 1)^{d - l - 1} \hat{b}_j^T\adj\left(xI - \hat{\Omega}\right)\hat{b}_j\\
                &= \frac{1}{(x - 1)^{1 - \delta}}\left[xp(x) - \sum_{i = 1}^{l} B_{ij}^2 x g_i(x) - \norms{c_j}^2p(x)\right]\,.
            \end{align*}
            In that case $p_j(x) = p_{\cS + j}(x)$, and the indices of the target coefficients already coincide.
        \end{enumerate}
        
        \item Now we apply Lemma~\ref{lem:vol_derand_lim}. At every step where it is applicable, i.e., where $m - r \leq k - t$, it guarantees that each computed $p_{\cS + j}$ is indeed a polynomial, and states that the ratio of its coefficients relates to the expected error:
        \[
        \E_{\cS + j}(A) = \varepsilon^2\left\vert\frac{c_{d - n + k + 1}(p_{j})}{c_{d - n + k}(p_{j})}\right\vert + O(\varepsilon^4)\,.
        \]
        By the properties of conditional expectations, $\min_{j \in \cR} \E_{\cS + j}(A) \leq \E_\cS(A)$. Therefore, choosing the index $s$ that minimizes the ratio $\vert c_{d - n + k + 1}(p_{\cS + j}) / c_{d - n + k}(p_{\cS + j})\vert$ ensures $\E_{\cS + s}(A) \leq \E_\cS(A) + O(\varepsilon^4)$.

        Observe that if $m - \rank(Q_{\cS + s}) > k - t - 1$, adding $k - t - 1$ columns to $\cS + s$ cannot result in a full-rank selection. By Proposition~\ref{stt:full_rank}, $\E_{\cS + s}(A) \geqslant {\rm const} > \E_{\cS}(A)$ for sufficiently small $\varepsilon$, which contradicts $\E_{\cS + s}(A) \leq \E_\cS(A) + O(\varepsilon^4)$. Therefore, $m - r \leq k - t$ is satisfied at all steps, making the algorithm well-defined and allowing for repeated application of Lemma~\ref{lem:vol_derand_lim}.
        
        After $k$ steps, we obtain the set $\cS$ of cardinality $k$ satisfying
        \[
        \normf{A - \bP_{A_\cS}A}^2 = \E_\cS(A) \leq \E_\varnothing(A) + k O(\varepsilon^4)\,.
        \]
        By Propositions~\ref{stt:frob_error_exp} and~\ref{stt:full_rank}, $Q_\cS$ has full rank and satisfies
        \[
        \normf{Q_\cS^\dag}^2 \leq m\frac{n - m + 1}{k - m + 1} + O(\varepsilon^2)
        \]
        for arbitrary $\varepsilon > 0$. Observing that $\normf{Q_\cS^\dag}^2 = \normf{Q_\cS^\dag Q}^2 = \normf{X_\cS^\dag L^{-1}LQ}^2 = \normf{X_\cS^\dag X}^2$ and taking the limit $\varepsilon \to 0$ yields the desired bound.
    \end{enumerate}

    \textbf{Proof of complexity.} The computational cost is dominated by the initial $\matdec{LQ}$ decomposition, which requires $O(nm^2)$ operations, and the subsequent $k$ iterations of the main loop. In each iteration, we perform the following steps:
    
    \begin{enumerate}
        \item{Auxiliary polynomial construction:} Constructing the coefficients of $\prod_{i = 1}^l (x - \omega_i)$ from its $l \leq m$ roots requires $O(m^2)$ operations. The subsequent computation of auxiliary polynomials $\prod_{q = 1,\,q\neq i}^l (x - \omega_q)$ by synthetic division also takes $O(m^2)$ in total; the division must be carried out from both ends, for the reasons given in Remark~\ref{rem:fdvs_stability}. Computing three coefficients up to a constant for $p(x)$ and $g_i(x)$ then takes no more than $O(m^2)$ operations in total, since it involves convolution with consecutive binomial coefficients computed up to a constant. 

        \item{Characteristic polynomial construction:}\label{itm:char_poly_construction} In the case $\delta = 1$, the summation in~\eqref{eqn:p_j_computation} involves summing $l \le m$ scalar values for each tracked coefficient. Thus, the cost is $O(m)$ per candidate, or $O(nm)$ for the entire set $\cR$. 

        If $\delta = 0$, it follows from the proof of bounds that the polynomial in square brackets in~\eqref{eqn:p_j_computation} has a root equal to $1$. Substituting $x = 1$ into the expression yields
        \begin{align*}
        &(1 - \norms{c_j}^2)\prod_{i = 1}^l(1 - \omega_i) - \sum_{i = 1}^{l} B_{ij}^2 \prod_{q = 1,\,q\neq i}^l(1 - \omega_q) = 0\,,\\
        &\norms{c_j}^2 = 1 - \sum_{i = 1}^{l} \frac{B_{ij}^2}{1 - \omega_i}\,.
        \end{align*}
        Substituting the above into~\eqref{eqn:p_j_computation} gives
        \begin{align*}
        p_{j}(x) &= \frac{1}{x - 1}\left[xp(x) - \sum_{i = 1}^{l} B_{ij}^2 x g_i(x) - \left(1 - \sum_{i = 1}^l \frac{B_{ij}^2}{1 - \omega_i}\right)p(x)\right]\\
        &=p(x) + \frac{1}{x - 1}\sum_{i = 1}^{l}B_{ij}^2 p(x)\left[\frac{1}{1 - \omega_i} - \frac{x}{x - \omega_i}\right]\\
        &= p(x) + \sum_{i = 1}^{l}\frac{B_{ij}^2 \omega_i}{1 - \omega_i}g_i(x)\,,
        \end{align*}
        which allows computing the coefficients of $p_{j}$ from the coefficients of $p(x)$ and $g_i(x)$ in $O(m)$ per index $j \in \cR$.

        \item{Greedy selection and update:}\label{itm:remaining_col_rotation} The rank-one eigendecomposition update of $\Lambda + q_s q_s^T$ takes $O(m^2)$~\cite{doi:10.1137/S089547989223924X}. Furthermore, for the deflated eigenvalue problem, the eigenvector $\hat{u}_i$ corresponding to $\hat{d}_i$ satisfies
        \[
        \hat{u}_i \propto \left(\frac{(\hat{q}_s)_1}{\hat{\lambda}_1 - \hat{d}_i}, \dots, \frac{(\hat{q}_s)_m}{\hat{\lambda}_m - \hat{d}_i}\right)\,,
        \]
        which makes the eigenvector matrix $\hat{U}$ Cauchy-like. Indeed, 
        \[
        \hat{U} = \diag(\hat{q}_s) C \diag(\gamma)\,,\ \text{where}\ C_{ij} = \frac{1}{\hat{\lambda}_i - \hat{d}_j}\,.
        \]
        Multiplication of $Q_\cR$ by such a matrix requires $O(nm)$ operations~\cite{doi:10.1137/S0895479893251472} with a hidden constant polylogarithmic in the target precision.
    \end{enumerate}
    
    Combining all the steps, we obtain a total complexity of $O(nkm)$.\qedhere
\end{proof}

\begin{remark}\label{rem:fdvs_stability}
Three details in Algorithm~\ref{alg:fdvs} can lead to instability or overflow when implemented carelessly.

The first issue is the construction of the polynomials $\hat{p}(x) = \prod_{i = 1}^l(x - \omega_i)$ and $\hat{g}_i(x) = \hat{p}(x)/(x - \omega_i)$. The former can be built directly by sequential multiplication: all its roots are positive, its coefficients alternate in sign, and multiplication by $(x - \omega_i)$ combines two like-signed terms. The quotients $\hat{g}_i$ are another matter. Write $\omega^{(i)}$ for $\omega$ with $\omega_i$ removed, so that the coefficient of $x^{l - 1 - s}$ in $\prod_{q \neq i}(x - \omega_q)$ is $(-1)^s e_s(\omega^{(i)})$. Forward synthetic division by $(x - \omega_i)$ unrolls to
\[
e_s\left(\omega^{(i)}\right) = \sum_{j = 0}^{s} (-\omega_i)^j\, e_{s - j}(\omega)\,,
\]
a sum whose terms alternate in sign. For $s \leq i$, each $e_{s - j}(\omega)$ is dominated by products of $\omega_1, \dots, \omega_i$, all of which are at least $\omega_i$; the terms therefore do not grow, the one with $j = 0$ dominates, and the sum is evaluated to full relative accuracy. For $s > i$ the terms instead grow to order $\omega_i^{s - i}\omega_1 \cdots \omega_i$, far above the result, and the accuracy is lost. For the backward recurrence, the situation is the opposite: the coefficients with $s \geq i$ are computed accurately. We therefore run the forward recurrence for $s \leq i$ and the backward one for $s > i$.

The second issue is that the coefficients of $\hat{p}$ and $\hat{g}_i$ decay to $\prod_{i = 1}^l \omega_i$ while the binomial coefficients of $(x - 1)^{d - l - \delta}$ grow, and for $l$ of a few hundred both leave the exponent range of double precision although their products do not. We therefore carry every coefficient already scaled by the binomial coefficient it will eventually meet, propagating only the ratios $w(a)/w(a + 1)$ of consecutive weights.

Lastly, the stability of the rank-one update $\Lambda + q_s q_s^T = UDU^T$ is established in~\cite{doi:10.1137/S089547989223924X}; the tolerance $\tau$ is set to the square root of machine epsilon.
\end{remark}

To see why Algorithm~\ref{alg:fdvs} indeed constitutes a forward derandomization of volume sampling, observe that the probability of selecting the subset $\cS$ in volume sampling for CSSP is proportional to $\det(A_\cS^T A_\cS)$. Substituting the structure of $A$, we find:
\[
    \det\left(A_\cS^T A_\cS\right) = \det\left(Q_\cS^T Q_\cS + \varepsilon^2 I\right) = \varepsilon^{2(k - m)}\det\left(Q_\cS Q_\cS^T\right) + O\left(\varepsilon^{2(k - m + 1)}\right)\,.
\]
Taking the limit $\varepsilon \to 0$ and recalling that $\det(X_\cS X_\cS^T) = \det(LL^T) \det(Q_\cS Q_\cS^T)$, we recover the standard volume sampling distribution, noting that the selection of a rank-deficient $Q_\cS$ has vanishing probability. Furthermore, the reconstruction error objective $\normf{A - \bP_{A_\cS}A}^2$ transforms into $\varepsilon^2(\normf{Q_\cS^\dag Q_\cR}^2 + n - k) + O(\varepsilon^4) = \varepsilon^2(\normf{X_\cS^\dag X}^2 + n - k - m) + O(\varepsilon^4)$, an affine function of the objective (assuming $Q_\cS$ has full rank; see~\eqref{eqn:C_final}).

\section{Forward iterative volume sampling algorithm}\label{sec:fivs}

In this section, we describe a very simple way to sample column subsets of cardinality $k \geq m$ with probability proportional to their squared volume. All of its ingredients are due to Derezi\'nski, Warmuth et al.~\cite{derezinski2022unbiased, pmlr-v99-derezinski19b}, although they never state the resulting algorithm explicitly, and, to our knowledge, it is absent from the numerical linear algebra literature. The key observation is Theorem~2.4 in~\cite{derezinski2022unbiased}, which states that a volume sample of cardinality $k$ decomposes into a volume sample of cardinality $m$ and $k - m$ columns drawn uniformly at random, the two being independent.

Fortunately, volume sampling for subsets of cardinality $m$ is a well-studied problem with many efficient solutions available~\cite{arp, pmlr-v99-derezinski19b, barthelme2023fastersamplerdiscretedeterminantal}. We present the resulting method in Algorithm~\ref{alg:forward_vol_sample} and establish its properties in Theorem~\ref{thm:forward_vol_sample}.

\begin{algorithm}[htbp]
\caption{Forward iterative volume sampling (FIVS).}\label{alg:forward_vol_sample}
\begin{algorithmic}[1]
\Require $X \in \R^{m \times n}$ ($m \leq n$, $\rank X = m$), sampling parameter $k \in \overline{m, n}$.
\Ensure Subset of column indices $\cS \subset \overline{1,n}$ with $|\cS| = k$.

\State Compute $\matdec{LQ}$ decomposition $X = LQ$. Save matrix $Q$.
\State Run Algorithm~2.1 in~\cite{arp} (or its accelerated version from~\cite{pmlr-v99-derezinski19b} with rejection sampling) on the matrix $Q$ to sample $\cS$ of cardinality $m$.
\State Choose $k - m$ distinct indices from $\cR = \overline{1, n} \setminus \cS$ uniformly at random; denote them by $\cT$. Set $\cS \gets \cS \cup \cT$.
\State \textbf{return} $\cS$
\end{algorithmic}
\end{algorithm}

\begin{theorem}\label{thm:forward_vol_sample}
    There exists a randomized algorithm (Algorithm~\ref{alg:forward_vol_sample}) that, given a full-rank matrix $X \in \R^{m \times n}$ with $m < n$ and a sampling parameter $k \in \overline{m, n}$, constructs a subset $\cS \subseteq \overline{1,n}$ of cardinality $k$. The submatrix $X_\cS$ has full rank, and
    \begin{gather}
    \E\left(\normf{X_\cS^\dag}^2\right) = \frac{n - m + 1}{k - m + 1}\normf{X^\dag}^2\,, \quad \E\left(\norms{X_\cS^\dag}^2\right) \leq \left( 1 + \frac{m(n - k)}{k - m + 1} \right)\norms{X^\dag}^2\,,\label{eqn:X_S_dag_vol_sample}\\
    \E\left(\normf{X_\cS^\dag X}^2\right) = m\frac{n - m + 1}{k - m + 1}\,, \quad \E\left(\norms{X_\cS^\dag X}^2\right) \leq \left( 1 + \frac{m(n - k)}{k - m + 1} \right)\,.\label{eqn:X_S_dag_X_vol_sample}
    \end{gather}
    The asymptotic complexity of the algorithm is $O(nm^2)$ ($O(nm + m^3 \log m)$ in expectation if $Q$ is known and the rejection sampler from~\cite{pmlr-v99-derezinski19b} is used).
\end{theorem}

\begin{proof}
     \textbf{Proof of bounds.} Note that $\vol(X_\cS) = \vol(L)\vol(Q_\cS)$; therefore, performing volume sampling on $X$ is equivalent to performing it on $Q$. Hence, by Theorem~2.4 in~\cite{derezinski2022unbiased}, Algorithm~\ref{alg:forward_vol_sample} constitutes valid volume sampling.

     The Frobenius norm bound for $\normf{X_\cS^\dag}^2$ follows from~\cite[Lemma 3.9]{faster_subset_selection} (see also~\cite[Equation~2.2]{derezinski2018reverse}). The bound for $\normf{X_\cS^\dag X}^2 = \normf{Q_\cS^\dag}^2$ follows by applying the same result to the matrix $Q$, which is valid by the equivalence noted above.

     To obtain the spectral norm bounds, let $\rho = \normf{X_\cS^\dag}^2 / \normf{X^\dag}^2$. Note that it satisfies $\sum_{i = 1}^m \sigma_i^2(X_\cS^\dag) = \rho \sum_{i = 1}^m \sigma_i^2(X^\dag)$. Moving the terms with indices $2, \dots, m$ to the right-hand side and using $\sigma_i(X_\cS^\dag) \geq \sigma_i(X^\dag)$, which also implies $\rho \geq 1$, we obtain
     \begin{equation}\label{eqn:spec_from_frob_bounds}
     \norms{X_\cS^\dag}^2 \leq \rho\,\sigma_1^2\left(X^\dag\right) + \sum_{i = 2}^m (\rho - 1)\,\sigma_i^2\left(X^\dag\right) \leq (m\rho - m + 1)\norms{X^\dag}^2\,.
     \end{equation}
     Taking expectations with $\E\rho = (n - m + 1)/(k - m + 1)$ yields the bound on $\E(\norms{X_\cS^\dag}^2)$. Applying~\eqref{eqn:spec_from_frob_bounds} to the matrix $Q$, for which $\norms{Q^\dag} = 1$ and $\normf{Q^\dag}^2 = m$, yields the bound on $\E(\norms{X_\cS^\dag X}^2)$.

     \textbf{Proof of complexity.} The algorithm consists of the $\matdec{LQ}$ decomposition, which costs $O(nm^2)$; Algorithm~2.1 in~\cite{arp} also runs in $O(nm^2)$; and uniform sampling of $k - m$ columns costs $O(k - m)$. The total complexity is therefore $O(nm^2)$. If $Q$ is precomputed, rejection sampling~\cite{pmlr-v99-derezinski19b} reduces the cost of the $m$-column sampling step to $O(nm + m^3 \log m)$ expected time.
\end{proof}

Note that FIVS compares favorably with RIVS~\cite{derezinski2018reverse}, the other exact volume sampling method for $k \geq m$. Both share the asymptotic complexity $O(nm^2)$ in the general case; however, FIVS attains it deterministically rather than in expectation, and with a smaller constant factor when $k \ll n$, compared to the rejection-sampling variant of RIVS. More importantly, when the orthonormal factor $Q$ is already available, FIVS runs in $O(nm + m^3 \log m)$ expected time via rejection sampling~\cite{pmlr-v99-derezinski19b}, whereas RIVS requires $O(nm^2)$ operations regardless of whether $Q$ is precomputed.

\section{Improving the reverse derandomized volume sampling algorithm}\label{sec:frob_improve}

In this section, we show that the well-known Algorithms~1 and~2 in~\cite{faster_subset_selection} in fact perform the derandomization of volume sampling for the $\normf{X_\cS^\dag}^2$ and $\normf{X_\cS^\dag X}^2$ objectives, respectively. We also propose modifications of these algorithms that improve their running time by removing the hidden logarithmic factor in the target precision.

The first claim follows from the fact (Equation~2.2 in~\cite{derezinski2018reverse}) that under volume sampling
\[
\E \left(\normf{X_\cS^\dag}^2\,|\,\cS \subseteq \overline{1, n}\,,\ \vert\cS\vert = k \right) = \frac{n - m + 1}{k - m + 1}\normf{X^\dag}^2\,.
\]
Applying the above to the submatrix $X_\cT$ for $\cT$ of cardinality $t$ yields
\[
\E\left(\normf{X_\cS^\dag}^2\ \vert\ \cS \subseteq \cT\,,\ \vert \cS \vert = k\right) = \frac{t - m + 1}{k - m + 1} \normf{X_\cT^\dag}^2\,.
\]
Therefore, greedily removing columns that minimize the Frobenius norm of the pseudoinverse (which is the procedure in Algorithm~1 in~\cite{faster_subset_selection}) is the same as minimizing the conditional expectation of the Frobenius norm of the pseudoinverse under volume sampling. Performing greedy removal on the orthonormal factor of $X$ (Algorithm~2 in~\cite{faster_subset_selection}) likewise constitutes the derandomization for the $\normf{Q_\cS^\dag}^2 = \normf{X_\cS^\dag X}^2$ objective.

We now proceed to the second question, namely, making Algorithms~1~and~2 in~\cite{faster_subset_selection} more computationally efficient. We start by employing the Sherman--Morrison formula. Let $X_\cS \in \R^{m \times (n - t)}$ (where $t \in \overline{0, k - 1}$ is the current iteration) be a rank-$m$ column submatrix of $X \in \R^{m \times n}$. By the Sherman--Morrison formula, $X_{\cS - j}$ has full rank (as is required for subset selection) if and only if $x_j^T (X_\cS X_\cS^T)^{-1}x_j < 1$, and
\begin{equation}\label{eqn:frob_pinv_recalc}
\begin{split}
    \normf{X_{\cS - j}^\dag}^2 &= \tr \left(X_{\cS - j}X_{\cS - j}^T\right)^{-1} = \tr \left(X_\cS X_\cS^T - x_jx_j^T\right)^{-1}\\
    &= \tr\left(\left(X_\cS X_\cS^T\right)^{-1} + \frac{\left(X_\cS X_\cS^T\right)^{-1}x_jx_j^T\left(X_\cS X_\cS^T\right)^{-1}}{1 - x_j^T\left(X_\cS X_\cS^T\right)^{-1}x_j}\right)\\
    &\overset{(*)}{=} \normf{X_\cS^\dag}^2 + \frac{x_j^T \left(X_\cS X_\cS^T\right)^{-2}x_j}{1 - x_j^T \left(X_\cS X_\cS^T\right)^{-1}x_j} = \normf{X_\cS^\dag}^2 + \frac{l_j^t}{1 - d_j^t}\,,\\
\end{split}
\end{equation}
where $d_j^t = x_j^T(X_\cS X_\cS^T)^{-1}x_j$, $l_j^t = x_j^T(X_\cS X_\cS^T)^{-2}x_j$, and
$(*)$ holds because of the cyclic property of the trace.

In~\cite{DEHOOG2007349, DEHOOG20111845}, it is proposed to evaluate $d_j^t$ and $l_j^t$ directly for each column, which requires $O(m^3)$ operations per column on each removal iteration. In~\cite{faster_subset_selection}, the authors employ the $\matdec{SVD}$ to bring this down to $O(m^2)$. This strategy, however, requires downdating the $\matdec{SVD}$ upon each column removal.

Our new method for computing $d_j^t$ and $l_j^t$ relies only on standard matrix operations and decompositions. We propose recalculating $l_j^t$ and $d_j^t$ from their values at the previous iteration. The following proposition formalizes this idea.

\begin{proposition}\label{stt:l_d_recalculation}
    Let $X \in \R^{m \times n}$ with $m < n$ be a full-rank matrix, and let $\cS \subseteq \overline{1, n}$ be a subset of cardinality $n - t \in \overline{m + 1, n - 1}$ such that $\rank X_\cS = m$. Then, for indices $j, s \in \overline{1, n}$ ($j \in \cS$, $s \notin \cS$), the scores $d_j^t$, $l_j^t$ can be obtained using
    \[
        d_j^t = d_j^{t - 1} + \frac{\left(a_j^{t - 1}\right)^2}{1 - d_s^{t - 1}}\,,\quad l_j^t = l_j^{t - 1} + \frac{a_j^{t - 1}}{1 - d_s^{t - 1}}\left(2 b_j^{t - 1} + \frac{a_j^{t - 1} b_s^{t - 1}}{1 - d_s^{t - 1}}\right)\,,
    \]
    where $d_s^{t - 1} = x_s^T u$, $l_s^{t - 1} = x_s^T v$, $a_j^{t - 1} = x_j^T u$, $b_j^{t - 1} = x_j^T v$ for vectors $u$ and $v$ defined by $u = (X_{\cS + s} X_{\cS + s}^T)^{-1}x_s$, $v = (X_{\cS + s} X_{\cS + s}^T)^{-2}x_s$.
\end{proposition}

\begin{proof}
    The proof proceeds by the Sherman--Morrison formula:
    \begin{align*}
    d_j^t &= x_j^T\left(X_{\cS}X_{\cS}^T\right)^{-1}x_j = x_j^T\left(X_{\cS + s} X_{\cS + s}^T - x_s x_s^T\right)^{-1}x_j\\
    &= x_j^T\left(\left(X_{\cS + s} X_{\cS + s}^T\right)^{-1} + \frac{\left(X_{\cS + s} X_{\cS + s}^T\right)^{-1}x_s x_s^{T}\left(X_{\cS + s} X_{\cS + s}^T\right)^{-1}}{1 - x_s^T \left(X_{\cS + s} X_{\cS + s}^T\right)^{-1} x_s } \right)x_j\\
    &= d_j^{t - 1} + \frac{\left(x_j^T\left(X_{\cS + s} X_{\cS + s}^T\right)^{-1}x_s \right)^2}{1 - x_s ^T \left(X_{\cS + s} X_{\cS + s}^T\right)^{-1} x_s } = d_j^{t - 1} + \frac{\left(a_j^{t - 1}\right)^2}{1 - d_s^{t - 1}}\,.
    \end{align*}

    Analogously,
    \begin{align*}
    l_j^t &= x_j^T\left(\left(X_{\cS + s} X_{\cS + s}^T\right)^{-1} + \frac{\left(X_{\cS + s} X_{\cS + s}^T\right)^{-1}x_s x_s^{T}\left(X_{\cS + s} X_{\cS + s}^T\right)^{-1}}{1 - x_s^T \left(X_{\cS + s} X_{\cS + s}^T\right)^{-1} x_s } \right)^2x_j\\
    &= l_j^{t - 1} + 2\frac{a_j^{t - 1} x_j^T\left(X_{\cS + s} X_{\cS + s}^T\right)^{-2}x_s}{1 - d_s^{t - 1}} + \frac{\left(a_j^{t - 1}\right)^2 x_s^T\left(X_{\cS + s} X_{\cS + s}^T\right)^{-1} x_s}{\left(1 - d_s^{t - 1}\right)^2}\\
    &= l_j^{t - 1} + \frac{a_j^{t - 1}}{1 - d_s^{t - 1}}\left(2 b_j^{t - 1} + \frac{a_j^{t - 1} b_s^{t - 1}}{1 - d_s^{t - 1}}\right)\,.\qedhere
    \end{align*}
\end{proof}

However, directly evaluating $(X_{\cS + s} X_{\cS + s}^T)^{-2}x_s$ can be unstable, so we suggest computing the $\matdec{SVD}$ of $X$, $X = U\Sigma V$ at the beginning of the algorithm, allowing the following transformations:
\begin{equation}\label{eqn:a_b_intermediate}
\begin{split}
    a_j^{t - 1} &= v_j^T \Sigma U^T U \Sigma^{-1} \left(V_{\cS + s} V_{\cS + s}^T\right)^{-1} \Sigma^{-1} U^T U v_s = v_j^T\left(V_{\cS + s} V_{\cS + s}^T\right)^{-1}v_s = v_j^T z_s^{t-1}\,,\\
    b_j^{t - 1} &= v_j^T \Sigma U^TU \Sigma^{-1} \left(V_{\cS + s} V_{\cS + s}^T\right)^{-1} \Sigma^{-2} \left(V_{\cS + s} V_{\cS + s}^T\right)^{-1} \Sigma^{-1} U^TU \Sigma v_s = z_j^T\Sigma^{-2}z_s\,,
\end{split}
\end{equation}
where $Z^{t - 1} = \left(V_{\cS + s}V_{\cS + s}^T\right)^{-1}V$. It can also be updated using the Sherman--Morrison formula upon column removal:
\begin{align}
    Z^t &= \left(V_\cS V_\cS^T\right)^{-1}V = \left(V_{\cS + s} V_{\cS + s}^T - v_s v_s^T\right)^{-1}V \label{eqn:z_recalculation}\\
    &= Z^{t - 1} + \frac{\left(V_{\cS + s} V_{\cS + s}^T\right)^{-1}v_s v_s^T \left(V_{\cS + s} V_{\cS + s}^T\right)^{-1}V}{1 - v_s^T\left(V_{\cS + s} V_{\cS + s}^T\right)^{-1} v_s}
    = Z^{t - 1} + \frac{z_s^{t - 1} \left(z_s^{t - 1}\right)^T Z^{t - 1}}{1 - d_s^{t - 1}}\,. \notag
\end{align}
Note that the additional computation of the $\matdec{SVD}$ changes the complexity of the algorithm to $O(nm^2 + nm(n - k))$ instead of $O(nm(n - k))$. The term $O(nm^2)$, however, is usually negligible compared to $O(nm(n - k))$ when sampling a practical number of columns ($k \ll n$). Moreover, the $\matdec{SVD}$ actually makes these steps a bit cheaper, so we recommend it.

The resulting algorithm is denoted as Algorithm~\ref{alg:frob_rem}, with its properties described in Theorem~\ref{thm:frob_rem}. This theorem is a restatement of Theorem~3.1 in~\cite{faster_subset_selection} with improved spectral norm bounds.

\begin{algorithm}[htbp]
\caption{Frobenius removal.}\label{alg:frob_rem}
\begin{algorithmic}[1]
\Require $X \in \R^{m \times n}$ ($m \leq n$, $\rank X = m$), sampling parameter $k \in \overline{m, n}$.
\Ensure Subset of column indices $\cS \subset \overline{1,n}$ with $|\cS| = k$.

\Phase{Initialization}
\State Set $\cS \gets \overline{1, n}$.
\State Compute $\matdec{SVD}$ of $X$, $X = U \Sigma V$ with $VV^T = I$, assign $Z^0 = V$.
\State \MultL{%
Compute initial values of $d$ and $l$ scores:
\[
d_i^0 \gets v_i^T z_i^0 \quad \text{and} \quad l_i^0 \gets \left(z_i^0\right)^T\Sigma^{-2}z_i^0 \quad \text{for}\ i \in \cS\,,
\]
where $v_j$ and $z_j^t$ denote the $j$-th columns of the matrices $V$ and $Z^t$, respectively.
}
\EndPhase
\For{$t = 0, \dots, n - k - 1$}   
    \Phase{Greedy selection}
    \State Find the index $s \in \cS$ that minimizes\footnote{As suggested by Avron and Boutsidis in~\cite{faster_subset_selection}, direct evaluation of this ratio is susceptible to catastrophic cancellation when $d_s^t \approx 1$. To maintain numerical stability, the comparison between two candidates $i$ and $j$ should be implemented as $l_i^t(1 - d_j^t) \le l_j^t(1 - d_i^t)$. Furthermore, the search should be restricted to indices satisfying $d_s^t < 1 - \tau$ for a small tolerance $\tau$ to avoid rank deficiency.} $l_s^t / (1 - d_s^t)$.
    \EndPhase
    
    \Phase{Update}
    \State Remove $s$ from $\cS$.
    \State Compute $a_j^t \gets v_j^T z_s^t$ and $b_j^t \gets (z_j^t)^T \Sigma^{-2} z_s^t$ for $j \in \cS$.
    \State \MultL{%
    Update $d_j$ and $l_j$ for all $j \in \cS$:
    \[
    d_j^{t + 1} \gets d_j^t + \frac{{a_j^t}^2}{1 - d_s^t}\,,\quad l_j^{t + 1} \gets l_j^t + \frac{a_j^t}{1 - d_s^t}\left(2b_j^t + \frac{a_j^tb_s^t}{1 - d_s^t}\right)\,.
    \]
    }\vspace{-1em}
    \State \MultL{%
    Update $Z$: 
    \[
    Z^{t + 1} \gets Z^t + \frac{z_s^t{z_s^t}^TZ^t}{1 - d_s^t}\,.
    \]
    }\vspace{-1em}
    \EndPhase
\EndFor
\State \textbf{return} $\cS$
\end{algorithmic}
\end{algorithm}

\begin{theorem}\label{thm:frob_rem}
    There exists a deterministic algorithm (Algorithm~\ref{alg:frob_rem}) that, given a full-rank matrix $X \in \R^{m \times n}$ with $m < n$ and a sampling parameter $k \in \overline{m, n}$, constructs a subset $\cS \subseteq \overline{1,n}$ of cardinality $k$. The algorithm ensures that $X_\cS$ has full rank, and
    \[
    \normf{X_\cS^\dag}^2 \leq \frac{n - m + 1}{k - m + 1}\normf{X^\dag}^2\,, \quad \norms{X_\cS^\dag}^2 \leq \left( 1 + \frac{m(n - k)}{k - m + 1} \right)\norms{X^\dag}^2\,.
    \]
    The asymptotic complexity of the algorithm is $O(nm^2 + nm(n - k))$ in the general case, and $O(nm(n - k))$ if the $\matdec{SVD}$ of $X$ is known.
\end{theorem}

\begin{proof}
    \textbf{Proof of bounds.} Algorithm~\ref{alg:frob_rem} is equivalent to Algorithm~1 in~\cite{faster_subset_selection}. Both algorithms greedily remove the column whose removal minimizes the Frobenius norm of the pseudoinverse of the remaining submatrix (for our algorithm, this follows by combining~\eqref{eqn:frob_pinv_recalc},~\eqref{eqn:a_b_intermediate}, and~\eqref{eqn:z_recalculation} with Proposition~\ref{stt:l_d_recalculation}). Consequently, the bounds established for the greedy removal strategy in~\cite[Theorem~3.1]{faster_subset_selection} apply directly to the subset $\cS$ constructed here, guaranteeing
    \[
    \normf{X_\cS^\dag}^2 \leq  \frac{n - m + 1}{k - m + 1}\normf{X^\dag}^2\,.
    \]
    
    The spectral norm bound follows by substituting $\rho = \normf{X_\cS^\dag}^2 / \normf{X^\dag}^2 \leq (n - m + 1)/(k - m + 1)$ into~\eqref{eqn:spec_from_frob_bounds}.

    \textbf{Proof of complexity.} The computational cost consists of the initialization phase and the iterative removal loop. In the initialization step, the main cost is the $\matdec{SVD}$, which takes $O(nm^2)$ operations. Computing the initial vectors $d^0$ and $l^0$ requires $O(nm)$ operations.

    The initialization step is followed by $n - k$ iterations. In each iteration $t$:
    
    \begin{enumerate}
        \item Finding the minimizer index $s$ requires scanning the active set $\cS$, taking $O(\vert\cS\vert) = O(n)$ operations.
        \item Computing the coefficients $a_i^t$ and $b_i^t$ for all $i \in \cS$ involves vector-vector operations of size $m$. This results in $O(mn)$ operations.
        \item Updating the scores $d$ and $l$ takes $O(n)$.
        \item Updating the matrix $Z$ costs $O(mn)$.
    \end{enumerate}
    
    Summing the loop costs yields $(n - k) \times O(mn) = O(nm(n - k))$. Combining this with the initialization, the total time complexity is $O(nm^2 + nm(n - k))$. If the $\matdec{SVD}$ is precomputed or $X$ has orthonormal rows, the dominant term is the loop, reducing the complexity to $O(nm(n - k))$.
\end{proof}

However, applied to a general full-rank $X$, the greedy procedure of Algorithm~\ref{alg:frob_rem} cannot guarantee any bounds on $\normxi{X_\cS^\dag X}$, $\xi \in \{2, F\}$, which, as we emphasized in the introduction, may be even more important. To provide the desired guarantees, we need to work with the orthonormal factor of $X$, which can be obtained via the $\matdec{LQ}$ decomposition. The resulting algorithm is denoted as Algorithm~\ref{alg:lq_frob_rem}.

\begin{algorithm}[htbp]
\caption{Reverse derandomized volume sampling (RDVS).}\label{alg:lq_frob_rem}
\begin{algorithmic}[1]
\Require $X \in \R^{m \times n}$ ($m \leq n$, $\rank X = m$), sampling parameter $k \in \overline{m, n}$.
\Ensure Subset of column indices $\cS \subset \overline{1,n}$ with $|\cS| = k$.

\State Compute $\matdec{LQ}$ decomposition $X = LQ$. Save matrix $Q$.
\State Run Algorithm~\ref{alg:frob_rem} with input matrix $Q$ and parameter $k$ to obtain $\cS$.
\State \textbf{return} $\cS$
\end{algorithmic}
\end{algorithm}

The properties of the subset selected by Algorithm~\ref{alg:lq_frob_rem} are summarized in the following corollary.

\begin{corollary}\label{cor:lq_bounds}
    Let $X \in \R^{m \times n}$ be a full-rank matrix, and let $\cS$ be the subset of indices returned by Algorithm~\ref{alg:lq_frob_rem}. Then $X_\cS$ has full rank, and the following bounds hold:
    \[
    \normf{X_\cS^\dag X}^2 \leq m\frac{n - m + 1}{k - m + 1}\,,\quad \norms{X_\cS^\dag X}^2 \leq 1 + \frac{m(n - k)}{k - m + 1}\,.
    \]
    Additionally,
    \[
    \norms{X_\cS^\dag}^2 \leq \left( 1 + \frac{m(n - k)}{k - m + 1} \right) \norms{X^\dag}^2\,.
    \]
    The asymptotic complexity of Algorithm~\ref{alg:lq_frob_rem} is $O(nm^2 + nm(n - k))$ in the general case, and $O(nm(n - k))$ if $X$ has orthonormal rows.
\end{corollary}

\begin{proof}
    First, recall that $X_\cS = L Q_\cS$. Since $L$ is invertible and Algorithm~\ref{alg:frob_rem} guarantees that $Q_\cS$ has rank $m$, $X_\cS$ must also have rank $m$. 

    Using $\normxi{X_\cS^\dag X}^2 = \normxi{Q_\cS^\dag}^2$, $\xi \in \{2, F\}$, the bounds from Theorem~\ref{thm:frob_rem} can be applied as follows:
    \begin{align*}
        \normf{X_\cS^\dag X}^2 &\leq \frac{n - m + 1}{k - m + 1}\normf{Q^\dag}^2 = \frac{m(n - m + 1)}{k - m + 1}\,,\\
        \norms{X_\cS^\dag X}^2 &\leq \left( 1 + \frac{m(n - k)}{k - m + 1} \right)\norms{Q^\dag}^2 = 1 + \frac{m(n - k)}{k - m + 1}\,.
    \end{align*}
    The bound on the spectral norm follows from
    \[
    \norms{X_\cS^\dag}^2 = \norms{Q_\cS^\dag L^{-1}}^2 \leq \norms{Q_\cS^\dag}^2 \norms{L^{-1}}^2 \leq \left(1 + \frac{m(n - k)}{k - m + 1}\right)\norms{X^\dag}^2\,.
    \]
    
    Algorithm~\ref{alg:lq_frob_rem} consists of the $\matdec{LQ}$ decomposition ($O(nm^2)$) and the execution of Algorithm~\ref{alg:frob_rem} on $Q$. Since the $\matdec{SVD}$ of $Q$ is trivial ($U=I$, $\Sigma=I$, $V=Q^T$), the latter costs $O(nm(n-k))$. Combining the two terms yields the overall complexity $O(nm^2 + nm(n - k))$ for Algorithm~\ref{alg:lq_frob_rem}.
\end{proof}

\section{Frobenius selection algorithm}\label{sec:frob_sel}

As demonstrated in Section~\ref{sec:frob_improve}, $\normf{X_\cS^\dag}^2$ can be cheaply updated upon column removal, which enables an efficient greedy algorithm. It is therefore tempting to apply the same approach to the greedy selection of columns.

Note that recalculations similar to~\eqref{eqn:frob_pinv_recalc} do not hold until $X_\cS$ has full rank, and a slightly different approach is required to select the first $m$ columns~\cite{SubSelQRfro}. A much more significant issue is that a greedy selection strategy is unable to provide tight bounds on the norm of the pseudoinverse of the selected submatrix.

To illustrate this, consider a matrix $X \in \R^{m \times (m + 1)}$ with columns defined as:
\[
x_1 = \frac{\alpha e_1 + (1 - \alpha) u}{\norms{\alpha e_1 + (1 - \alpha) u}}\,,\ \dots\,,\ x_m = \frac{\alpha e_m + (1 - \alpha) u}{\norms{\alpha e_m + (1 - \alpha) u}}\,,\ x_{m + 1} = u\,,
\]
where $e_1,\dots,e_m$ are the standard basis vectors, $u = \sum_{j = 1}^m m^{-1/2} e_j$, and $\alpha \in (0, 1)$ is a small parameter. Geometrically, the columns $x_1,\dots,x_m$ form a narrow cone around the central vector $x_{m + 1}$.

We can assume (or perturb the matrix slightly to ensure) that the greedy selection algorithm picks $x_{m + 1}$ at the first iteration. The symmetry of the remaining columns implies that any subsequent choice is equivalent. Thus, the greedy algorithm effectively selects the subset $\cS = \{2, \dots, m + 1\}$. In contrast, the optimal subset of size $m$, which is also the one chosen by Algorithm~\ref{alg:frob_rem}, is $\cS_{opt} = \{1, \dots, m\}$.

To evaluate the ratio $\normf{X_\cS^\dag}^2 / \normf{X^\dag}^2$, denote the common normalization factor $c = \norms{\alpha e_1 + (1 - \alpha) u}$ and observe that
\[
X = \frac{\alpha}{c}
\begin{pNiceMatrix}[hvlines, margin, cell-space-limits = 3pt]
    \Block{2-1}{I} & \Block{2-1}{0_{m \times 1}} \\
    &
\end{pNiceMatrix}
+ uv^T \quad \text{where} \quad v = 
\begin{pNiceMatrix}[margin, cell-space-limits = 3pt]
    \Block{1-1}{\frac{1 - \alpha}{c}} & \dots &\Block{1-1}{\frac{1 - \alpha}{c}} & 1
\end{pNiceMatrix}^T\,.
\]
Since the first $m$ entries of $v$ are equal, the corresponding part of $uv^T$ is a multiple of $uu^T$, and we obtain
\[
XX^T = \frac{\alpha^2}{c^2}I + \left(1 + \frac{m(1 - \alpha)^2}{c^2} + \frac{2\sqrt{m}\alpha(1 - \alpha)}{c^2}\right)uu^T = fI + guu^T\,.
\]
The expression for $(XX^T)^{-1}$ can be obtained by the Sherman--Morrison formula:
\[
\left(XX^T\right)^{-1} = \frac{1}{f}I - \frac{g}{f^2 + fg}uu^T\,.
\]
Finally, we employ~\eqref{eqn:frob_pinv_recalc} to compute the Frobenius norm of $X_\cS^\dag$. A careful series expansion then reveals
\[
\normf{X^\dag}^2 = \frac{m - 1}{\alpha^2} + O\left(\alpha^{-1}\right)\,, \quad 
\normf{X_\cS^\dag}^2 = \normf{X^\dag}^2 + \frac{m^2 - 1}{\alpha^2} + O\left(\alpha^{-1}\right)\,.
\]

The above means that the ratio $\normf{X_\cS^\dag}^2 / \normf{X^\dag}^2$ approaches $m + 2$ as $\alpha$ approaches~$0$, while for the greedy removal approach it is bounded by $(n - m + 1) / (k - m + 1) = 2$. The greedy selection approach thereby fails catastrophically.

The solution to this problem is to work with the orthonormal factor $Q$ of $X$. In this case, greedily selecting the first $m$ columns is equivalent to running Algorithm~2 in~\cite{OSINSKY2025}, which guarantees
\begin{equation}\label{eqn:CCA_guarantees}
    \normf{Q_\cS^\dag}^2 \leq m(n - m + 1) \quad \text{and} \quad \norms{Q_\cS^\dag}^2 \leq 1 + m(n - m)\,.
\end{equation}

To see that Algorithm~2 in~\cite{OSINSKY2025} indeed constitutes greedy selection in the Frobenius norm, observe that at each step the chosen column $s$ satisfies
\begin{equation}\label{eqn:CCA_update}
s \in \argmin_{j \in \cR}\frac{1 + \norms{Q_\cS^\dag q_j}^2}{\norms{(I - \bP_{Q_\cS})q_j}^2}
\end{equation}
in our notation, where $\cR = \overline{1, n} \setminus \cS$.

At the same time, we can use a generalization of the Sherman--Morrison formula~\cite[(8)]{doi:10.1137/1002004} to obtain
\begin{align*}
Q_{\cS + j}^\dag &= \begin{pNiceMatrix}[hvlines, margin, cell-space-limits = 3pt]
    \Block{1-1}{Q_\cS^\dag - Q_\cS^\dag q_j b^T}\\
    \Block{1-1}{b^T}
\end{pNiceMatrix}
\quad \text{where} \quad b = \frac{(I - \bP_{Q_\cS})q_j}{\norms{(I - \bP_{Q_\cS})q_j}^2}\,.
\end{align*}
Therefore 
\begin{align*}
\normf{Q_{\cS + j}^\dag}^2 &= \normf{Q_\cS^\dag - Q_\cS^\dag q_j b^T}^2 + \norms{b}^2\\
&\overset{(*)}{=} \normf{Q_\cS^\dag}^2 + \norms{Q_\cS^\dag q_j}^2\norms{b}^2 + \norms{b}^2\\
&=  \normf{Q_\cS^\dag}^2 + \frac{1 + \norms{Q_\cS^\dag q_j}^2}{\norms{(I - \bP_{Q_\cS})q_j}^2}\,,
\end{align*}
$(*)$ follows from $\langle Q_\cS^\dag, Q_\cS^\dag q_j b^T \rangle = \langle Q_\cS^\dag b, Q_\cS^\dag q_j \rangle = 0$.

After the first $m$ columns are chosen, the Frobenius norm of the pseudoinverse can be updated using the usual Sherman--Morrison formula. Similarly to~\eqref{eqn:frob_pinv_recalc}, we can derive
\[
    \normf{Q_{\cS + j}^\dag}^2 = \normf{Q_\cS^\dag}^2 - \frac{q_j^T \left(Q_\cS Q_\cS^T\right)^{-2}q_j}{1 + q_j^T \left(Q_\cS Q_\cS^T\right)^{-1}q_j} = \normf{Q_\cS^\dag}^2 - \frac{l_j^t}{1 + d_j^t}\,,\\
\]
where $d_j^t = q_j^T(Q_\cS Q_\cS^T)^{-1}q_j$ and $l_j^t = q_j^T(Q_\cS Q_\cS^T)^{-2}q_j$. These quantities can be updated at each step, as was done in Algorithm~\ref{alg:frob_rem}.

The resulting algorithm is denoted as Algorithm~\ref{alg:frob_sel}, with its properties described in Theorem~\ref{thm:frob_sel}. The bounds stated in the theorem are weaker than those that hold for the greedy removal strategy when $k > m$: they exceed the bound $m(n - m + 1)/(k - m + 1)$ of Table~\ref{tbl:alg_summary} by a factor of roughly $m$ when $m \ll k \ll n$. In exchange, the algorithm can be much faster than both FDVS and RDVS.

\begin{algorithm}[htbp]
\caption{Frobenius selection algorithm.}\label{alg:frob_sel}
\begin{algorithmic}[1]
\Require $X \in \R^{m \times n}$ ($m \leq n$, $\rank X = m$), sampling parameter $k \in \overline{m, n}$.
\Ensure Subset of column indices $\cS \subset \overline{1,n}$ with $|\cS| = k$.
\Phase{Initialization}
    \State Compute $\matdec{LQ}$ decomposition of $X$, save matrix $Q$.
    \State Run Algorithm~2 from~\cite{OSINSKY2025} on $Q$ to obtain the starting set of indices $\cS$.
    \State Set $\cR \gets \overline{1, n} \setminus \cS$, $Z^m \gets (Q_\cS Q_\cS^T)^{-1}Q$.
    \State Compute initial values of $d$ and $l$ scores:
    \[
    d_j^m \gets q_j^T z_j^m \quad \text{and} \quad l_j^m \gets \left(z_j^m\right)^Tz_j^m \quad \text{for}\ j \in \cR\,.
    \]
\EndPhase

\For{$t = m, \dots, k-1$}
    \Phase{Greedy selection}
    \State Find the index $s \in \cR$ that maximizes $l_s^t / (1 + d_s^t)$.
    \EndPhase
    
    \Phase{Update}
    \State Move $s$ from $\cR$ to $\cS$.
    \State Compute $a_j^t \gets q_j^T z_s^t$ and $b_j^t \gets (z_j^t)^T z_s^t$ for $j \in \cR$.
    \State \MultL{%
    Update $d_j$ and $l_j$ for all $j \in \cR$:
    \[
    d_j^{t + 1} \gets d_j^t - \frac{{a_j^t}^2}{1 + d_s^t}\,,\quad l_j^{t + 1} \gets l_j^t - \frac{a_j^t}{1 + d_s^t}\left(2b_j^t - \frac{a_j^tb_s^t}{1 + d_s^t}\right)\,.
    \]
    }\vspace{-1em}
    \State \MultL{%
    Update $Z$: 
    \[
    Z^{t + 1} \gets Z^t - \frac{z_s^t{z_s^t}^TZ^t}{1 + d_s^t}\,.
    \]
    }\vspace{-1em}
    \EndPhase
\EndFor
\State \textbf{return} $\cS$
\end{algorithmic}
\end{algorithm}

\begin{theorem}\label{thm:frob_sel}
    There exists a deterministic algorithm (Algorithm~\ref{alg:frob_sel}) that, given a full-rank matrix $X \in \R^{m \times n}$ with $m < n$ and a sampling parameter $k \in \overline{m, n}$, constructs a subset $\cS \subseteq \overline{1,n}$ of cardinality $k$. The algorithm ensures that $X_\cS$ has full rank, and
    \[
    \normf{X_\cS^\dag X}^2 \leq \frac{m^2}{k}(n - m + 1)\,, \quad \norms{X_\cS^\dag X}^2 \leq \frac{m^2}{k}(n - m + 1) - m + 1\,.
    \]
    The asymptotic complexity of the algorithm is $O(nkm)$. 
\end{theorem}

\begin{proof}
    \textbf{Proof of bounds.} First, note that the update formulas for the coefficients $d_j$, $l_j$ for $j \in \cR$ and the matrix $Z$ indeed guarantee that at iteration $t$ (when $|\cS| = t$),
    \[
    d_j^t = q_j^T\left(Q_\cS Q_\cS^T\right)^{-1}q_j\,,\quad l_j^t = q_j^T\left(Q_\cS Q_\cS^T\right)^{-2}q_j\,,\quad Z^t = \left(Q_\cS Q_\cS^T\right)^{-1}Q\,.
    \]
    We omit the detailed proof here, as it comes down to repeating Proposition~\ref{stt:l_d_recalculation} and~\eqref{eqn:z_recalculation} with column selection instead of column removal, which changes the signs in the Sherman--Morrison formula.

    Now we proceed to the actual proof of bounds. By Theorem~5 in~\cite{OSINSKY2025}, after the initialization phase, that is, at iteration $t = m$, the selected subset satisfies~\eqref{eqn:CCA_guarantees}. At each subsequent iteration,
    \[
    \normf{Q_{\cS + s}^\dag}^2 - \normf{Q_\cS^\dag}^2 = - \frac{l_s^t}{1 + d_s^t} \leq -\E\left(\frac{l_j^t}{1 + d_j^t}\right) \,,
    \]
    where the index $j$ is sampled proportionally to the squared volume of the matrix $Q_{\cS + j}$, $\det(Q_{\cS + j} Q_{\cS + j}^T)$. Using the matrix determinant lemma, it can be transformed as follows:
    \begin{equation}\label{eqn:col_add_expectation}
    \det\left(Q_{\cS + j} Q_{\cS + j}^T\right) = \det\left(Q_\cS Q_\cS^T\right)\left(1 + q_j^T \left(Q_\cS Q_\cS^T\right)^{-1}q_j\right)\,.
    \end{equation}
    
    The normalization factor in this probability distribution thus equals
    \begin{align*}
    C^t &= \sum_{j \in \cR}\det\left(Q_\cS Q_\cS^T\right)\left(1 + q_j^T \left(Q_\cS Q_\cS^T\right)^{-1}q_j\right)\\
    &\overset{(*)}{=} \det\left(Q_\cS Q_\cS^T\right)\left(n - t + \tr\left(\left(Q_\cS Q_\cS^T\right)^{-1}\left(I - Q_\cS Q_\cS^T\right)\right)\right)\\
    &= \det\left(Q_\cS Q_\cS^T\right)\left(n - m - t + \normf{Q_\cS^\dag}^2\right)\,,
    \end{align*}
    $(*)$ holds because of the cyclic property of the trace and the fact that $Q_\cS Q_\cS^T + Q_\cR Q_\cR^T = I$. Therefore,~\eqref{eqn:col_add_expectation} transforms into
    \[
    \normf{Q_{\cS + s}^\dag}^2 - \normf{Q_\cS^\dag}^2 \leq - \sum_{j \in \cR} \frac{q_j^T \left(Q_\cS Q_\cS^T\right)^{-2}q_j}{n - m - t +\normf{Q_\cS^\dag}^2} \overset{(*)}{\leq} -\frac{\normf{Q_\cS^\dag}^4 / m - \normf{Q_\cS^\dag}^2}{n - m - t + \normf{Q_\cS^\dag}^2}\,,
    \]
    $(*)$ holds because of the cyclic property of the trace and the sequence of inequalities $\normf{(Q_\cS Q_\cS^T)^{-1}}^2 = \sum_{i = 1}^m \sigma_i^{-4}(Q_\cS) \geq \norms{Q_\cS^\dag}^2\normf{Q_\cS^\dag}^2 \geq \normf{Q_\cS^\dag}^4 / m$. Denoting $x_t = \normf{Q_\cS^\dag}^2$ at iteration $t$, we obtain
    \begin{equation}\label{eqn:pinv_recursion}
        x_{t + 1} \leq x_t\left(1 - \frac{(x_t - m)/m}{n - t + (x_t - m)}\right)\,.
    \end{equation}

    We begin analyzing this recurrence by noting that the expression on the right-hand side is a monotonically increasing function of $x_t$ for $x_t \geq m$. Substituting $x_t = m + \alpha(n - t)$ with $\alpha \geq 0$ and collecting the $\alpha$-dependent terms over the common denominator $1 + \alpha$, the right-hand side becomes exactly
    \[
    m + \alpha\frac{(n - t - 1) + \frac{m - 1}{m}(n - t)\alpha}{1 + \alpha} = m + \frac{b\alpha + c\alpha^2}{1 + \alpha}\,,
    \]
    where $b = n - t - 1$ and $c = (m - 1)(n - t)/m$. Computing the derivative of that expression we obtain
    \[
    \frac{b + 2c\alpha + c\alpha^2}{(1 + \alpha)^2} \geq \frac{n - t - 1}{(1 + \alpha)^2} > 0 \quad \text{for} \quad t < n - 1\,,
    \]
    since $b \geq 0$ and $c \geq 0$. Hence the right-hand side is strictly increasing in $\alpha$, and therefore in $x_t$. Consequently, we can analyze an upper bound on $x_t$ instead of $x_t$ itself.

    We claim that $x_t \leq y_t = m^2(n - m + 1) / t$. For $t = m$, this holds by~\cite[Theorem~5]{OSINSKY2025}. Now we prove the induction step. Observe that for any fixed $x_t$ outside the brackets, the expression in brackets on the right-hand side of~\eqref{eqn:pinv_recursion} decreases as $x_t \ge m$ grows. We can thus replace the upper bound $y_t$ with a smaller value there:
    \[
    x_{t + 1} \leq y_t\left(1 - \frac{(\beta - m)/m}{n - t + (\beta - m)}\right)\,, \quad \text{where} \quad \beta = m\frac{n - m + 1}{t - m + 1} \leq y_t\,.
    \]
    Transforming this further, we obtain
    \begin{align*}
    x_{t + 1} &\leq \frac{m^2}{t}(n - m + 1)\left(1 - \frac{\frac{n - t}{t - m + 1}}{n - t + m\frac{n - t}{t - m + 1}} \right)\\
    &= \frac{m^2}{t}(n - m + 1) \left(1 - \frac{1}{t + 1}\right) = y_{t + 1}\,.
    \end{align*}
    The Frobenius norm bound now follows from $\normf{X_\cS^\dag X}^2 = \normf{Q_\cS^\dag}^2 = x_k \leq m^2(n - m + 1)/k$.

    The spectral norm bound follows by applying~\eqref{eqn:spec_from_frob_bounds} to the matrix $Q$, for which $\norms{Q^\dag} = 1$ and $\normf{Q^\dag}^2 = m$, with $\rho = \normf{Q_\cS^\dag}^2 / m \leq m(n - m + 1)/k$.

    \textbf{Proof of complexity.} The initialization step requires the $\matdec{LQ}$ decomposition ($O(nm^2)$), running Algorithm~2 in~\cite{OSINSKY2025} ($O(nm^2)$), computing $Z^m$ ($O(nm^2)$), and calculating the initial values of $d$ and $l$, which requires $O(n)$ scalar products of length $m$ and thus $O(nm)$ operations. The asymptotic complexity of each iteration is identical to that of Algorithm~\ref{alg:frob_rem}, namely $O(nm)$ operations. Combining the cost of $k - m$ iterations with the initialization complexity, we obtain $O(nkm)$.
\end{proof}

\begin{remark}\label{rem:fdvs_frob_sel_equivalence}
    For $k = m$, FDVS is equivalent to Frobenius selection. In the notation of Algorithm~\ref{alg:fdvs}, every step must increase the rank, so $\rank Q_{\cS + s} = t + 1$ and FDVS minimizes $|c_1(p_{\cS + s}) / c_0(p_{\cS + s})|$. Since $p_{\cS + s}(x) = (x - 1)^{n - m - t - 1}\tilde{p}_{\cS + s}(x)$, where $\tilde{p}_{\cS + s}(x) = \prod_{i = 1}^{t + 1}(x - \sigma_i^2(Q_{\cS + s}))$ has positive roots,
    \[
        \left\vert\frac{c_1(p_{\cS + s})}{c_0(p_{\cS + s})}\right\vert = \left\vert\frac{c_1(\tilde{p}_{\cS + s})}{c_0(\tilde{p}_{\cS + s})}\right\vert + n - m - t - 1 = \normf{Q_{\cS + s}^\dag}^2 + n - m - t - 1\,.
    \]
    The shift does not depend on $s$, so FDVS greedily minimizes $\normf{Q_{\cS + s}^\dag}^2$, which is the selection rule~\eqref{eqn:CCA_update} of Algorithm~2 in~\cite{OSINSKY2025}.
\end{remark}

\section{Numerical experiments}\label{sec:experiments}

Beyond theoretical guarantees, it is important to test the algorithms in practical scenarios to determine whether better bounds translate into better performance, which is not always the case.

For completeness, we also test several methods that were left out of Table~\ref{tbl:alg_summary} due to their weak Frobenius norm bounds, but that accept arbitrary $k \in \overline{m, n}$, are fast, and are known to perform well in practice. The complete list of tested algorithms is given in Table~\ref{tab:algs_experiments}.

\begin{table}[htbp]
\renewcommand{\arraystretch}{1.15}
\resizebox{\linewidth}{!}{%
\begin{NiceTabular}{|l|l|c|l|}
\hline
Codename & Method & Complexity & Note \\
\hline

\Block{2-1}{FDVS} & \Block{2-1}{Algorithm~\ref{alg:fdvs}}  & \Block{2-1}{$O(nkm)$} & \Block{2-1}{$O(nkm^2)$ complexity in practice}\\
& & & \\
\hline

\Block{2-1}{RDVS} & \Block{2-1}{Algorithm~\ref{alg:lq_frob_rem}}  & \Block{2-1}{$O(nm^2 + nm(n - k))$} & \\
& & & \\
\hline

\Block{2-1}{Frobenius removal} & \Block{2-1}{Algorithm~\ref{alg:frob_rem}}  & \Block{2-1}{$O(nm^2 + nm(n - k))$} & \\
& & & \\
\hline

\Block{2-1}{Frobenius selection} & \Block{2-1}{Algorithm~\ref{alg:frob_sel}}  & \Block{2-1}{$O(nkm)$} & \\
& & & \\
\hline

\Block{2-1}{Dominant} & \Block{2-1}{Dominant~\cite{Osinsky23}}  & \Block{2-1}{$O(nkm\log_c k)$} & \Block{2-1}{$c = 1 + 10^{-6}$, greedy initialization~\cite{osinsky2026subsetselectionmatricescolumn}, \\ $O(nkm)$ complexity in practice}\\
& & & \\
\hline

\Block{2-1}{Dominant-split} & \Block{2-1}{Algorithm~1 in~\cite{osinsky2026subsetselectionmatricescolumn}} & \Block{2-1}{$O(nkm + nm^2 \log\log m$ \\ $+\, nm^2 / \log c)$} & \Block{2-1}{$c = 1 + 10^{-6}$, greedy initialization~\cite{osinsky2026subsetselectionmatricescolumn}, \\ $O(nm^2)$ complexity in practice for small $k$}\\
& & & \\
\hline

\Block{2-1}{Spectral selection} & \Block{2-1}{Algorithm~4 in~\cite{subset_spectral}} & \Block{2-1}{$O(nkm)$\footnote{Stated complexity follows from applying the method from Item~\ref{itm:remaining_col_rotation} of the proof of Theorem~\ref{thm:vol_derand} to the analogous step in Algorithm~4 in~\cite{subset_spectral}}} & \Block{2-1}{$O(nkm^2)$ complexity in practice}\\
& & & \\
\hline

\Block{2-1}{GappyPOD+E} & \Block{2-1}{Algorithm~1 in~\cite{doi:10.1137/19M1307391}} & \Block{2-1}{$O(nkm)$\footnote{Stated complexity follows from replacing full $\matdec{SVD}$ of the selected submatrix with a rank-one update~\cite{doi:10.1137/S089547989223924X} on each iteration.}} & \Block{2-1}{$O(nkm + km^3)$ complexity\\ in our implementation} \\
& & & \\
\hline

\Block{2-1}{VS} & \Block{2-1}{Algorithm~\ref{alg:forward_vol_sample}} & \Block{2-1}{$O(nm^2)$} & \Block{2-1}{Implemented as FIVS}\\
& & & \\
\hline

\Block{2-1}{Leverage scores} & \Block{2-1}{Leverage scores~\cite{1109557.1109682}} & \Block{2-1}{$O(nm^2)$} & \\
& & & \\
\hline

\end{NiceTabular}%
}
\caption{Summary of approximation algorithms for Problem~\ref{prm:subset_selection} compared in our numerical experiments. All algorithms accept arbitrary $k \in \overline{m, n}$.}
\label{tab:algs_experiments}
\end{table}

The algorithms are tested on matrices arising in optimal experimental design, sensor placement, and DLRA-DEIM. The main quantity measured in our experiments is $\normf{X_\cS^\dag X} / \sqrt{m}$, but we additionally record $\normf{X_\cS^\dag} / \normf{X^\dag}$, $\norms{X_\cS^\dag X}$, $\norms{X_\cS^\dag} / \norms{X^\dag}$, $(\vol(X) / \vol(X_\cS))^{1/m}$, and the target metric of each application.

Some algorithms from Table~\ref{tab:algs_experiments} coincide or admit natural variants depending on whether $X$ has orthonormal rows and whether $k = m$. We handle these cases in the figures as follows.
\begin{enumerate}
    \item\label{itm:comma} For matrices with orthonormal rows, Frobenius removal is equivalent to RDVS, so we plot a single line or marker for both. In the legend, such coinciding algorithms are separated by a comma.
    \item\label{itm:comma_2} For $k = m$, FDVS and Frobenius selection are equivalent and are likewise separated by a comma.
    \item\label{itm:comma_3} For $k = m$, GappyPOD+E~\cite{doi:10.1137/19M1307391} reduces to QDEIM~\cite{deim}, since its greedy procedure only starts after the first $m$ columns are selected; we again use a comma.
    \item\label{itm:vertical_bar} For matrices with nonorthonormal rows, RDVS performed strictly better than Frobenius removal in terms of $\normxi{X_\cS^\dag X}$, $\xi \in \{F, 2\}$, in our experiments, while Frobenius removal was better in terms of $\normxi{X_\cS^\dag} / \normxi{X^\dag}$. We therefore display RDVS on the subplots for the first group of objectives and Frobenius removal on those for the second group, and separate the two by \enquote{|} in the legend.
    \item\label{itm:vertical_bar_2} As formulated, Frobenius selection and GappyPOD+E operate on an orthonormal basis of the row space of $X$. However, their greedy steps, which start after the first $m$ columns are chosen, can be applied to $X$ directly. In our experiments, this \enquote{no LQ} variant yields strictly better results in terms of $\normxi{X_\cS^\dag} / \normxi{X^\dag}$, $\xi \in \{F, 2\}$, so we display it on the corresponding subplots, again separated by \enquote{|} in the legend.
\end{enumerate}

The complete code used to conduct the experiments is available at \url{https://github.com/KozyrevIN/subset-selection-for-matrices}. The algorithms are implemented in C++ using the Eigen library, and the figures are produced in Python with Matplotlib.

\subsection{Running time comparison}\label{sec:runtime}

We timed the algorithms of Table~\ref{tab:algs_experiments} on random Gaussian $m \times n$ matrices over the grid $m \in \{32, 64, 128\}$, $n \in \{2^{11}, \dots, 2^{14}\}$, $k \in \{m, 2m, 3m, 4m\}$. All runs were single-threaded on an AMD Ryzen 7 8845H CPU warmed up to its boost frequency with 32\,GB of DDR5 RAM. Figure~\ref{fig:runtime} and Table~\ref{tab:runtime_constants} show the running times fitted as $t = C f$, where $f$ is the leading complexity term. No grid point deviates from its fit by more than a factor of $4.01$.

The cost models follow the practical complexities of Table~\ref{tab:algs_experiments}. For Dominant, we drop the $\log_c k$ factor, and for Dominant-split, we keep only the $nm^2$ term, which fits best: with greedy initialization, the number of swaps for both algorithms is small and roughly constant~\cite{osinsky2026subsetselectionmatricescolumn}. For GappyPOD+E, we keep only the $nkm$ term, which leaves the largest deviation in the table. RDVS and Frobenius removal have indistinguishable timings, so only the RDVS fit is listed. LQ is the thin $\matdec{LQ}$ decomposition alone, computed with \texttt{Eigen::HouseholderQR} applied to $X^T$; it is a natural baseline, since most algorithms start with it.

\begin{figure}[htbp]
\centering
\includegraphics[width=\textwidth]{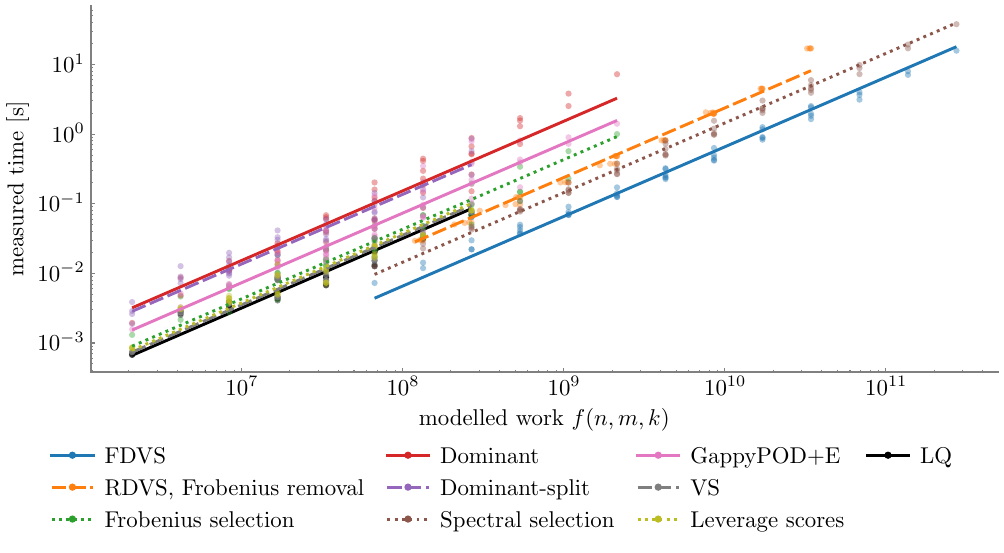}
\caption{Measured running time against the cost model $f$ of Table~\ref{tab:runtime_constants}. Markers are grid points; lines are the fits $t = C f$.}
\label{fig:runtime}
\end{figure}

\begin{table}[htbp]
\centering
\resizebox{\linewidth}{!}{%
\begin{tabular}{|l|c|c|c|}
\hline
Codename & Cost model $f$ & $C$ [$10^{-11}$\,s] & Maximum deviation \\
\hline
FDVS & $nkm^2$ & $6.55 \pm 0.43$ & $1.70\times$ \\
\hline
RDVS, Frobenius removal & $nm(n-k)$ & $24.2 \pm 1.8$ & $2.18\times$ \\
\hline
Frobenius selection & $nkm$ & $42.9 \pm 5.6$ & $2.59\times$ \\
\hline
Dominant & $nkm$ & $153 \pm 23$ & $2.71\times$ \\
\hline
Dominant-split & $nm^2$ & $136 \pm 16$ & $2.35\times$ \\
\hline
Spectral selection & $nkm^2$ & $14.4 \pm 0.7$ & $1.41\times$ \\
\hline
GappyPOD{+}E & $nkm$ & $73.1 \pm 11.0$ & $4.01\times$ \\
\hline
VS & $nm^2$ & $35.2 \pm 3.4$ & $1.87\times$ \\
\hline
Leverage scores & $nm^2$ & $37.5 \pm 4.1$ & $2.14\times$ \\
\hline
LQ & $nm^2$ & $31.7 \pm 3.4$ & $2.19\times$ \\
\hline
\end{tabular}%
}
\caption{Runtime constants fitted as $t = C f$. $C$ is the geometric mean of $t / f$ over the $48$ grid points of Section~\ref{sec:runtime}, with a 95\% Student-$t$ interval reflecting its scatter across the grid. Maximum deviation is the largest factor by which a grid point deviates from $C f$.}
\label{tab:runtime_constants}
\end{table}

By the fitted constants, Frobenius selection costs about $1.35\,k/m$ LQ decompositions. For $k$ close to $m$, it thus costs little more than the LQ decomposition itself and about as much as VS and Leverage scores ($1.11$ and $1.18$ LQ decompositions). It is the fastest deterministic algorithm for $k < 3.2m$; beyond that, Dominant-split, whose cost grows very slowly with $k$, is cheaper. FDVS is about twice as fast as Spectral selection, which shares its cost model.

\subsection{Optimal experimental design}

In experimental design, one estimates parameters $\beta \in \R^m$ from a linear model $y = X^T \beta + \varepsilon$, $\varepsilon \sim \mathcal{N}(0, \sigma^2I)$, where the columns of $X \in \R^{m \times n}$ are candidate experiments and a budget allows only $k$ of them to be run. Choosing a subset $\cS$ fixes the covariance of the least-squares estimator $\hat\beta = (X_\cS^\dag)^T y_\cS$ as $\sigma^2 (X_\cS X_\cS^T)^{-1}$. Classical optimality criteria are built from this covariance~\cite{Huan_Jagalur_Marzouk_2024, doi:10.1287/moor.2021.1129}. A-optimal design minimizes its trace, i.e., the total variance of $\hat\beta$, which is exactly $\sigma^2 \normf{X_\cS^\dag}^2$. V-optimal design instead minimizes the average variance of the predicted response $\hat y = X^T \hat\beta$ over the candidate points themselves, $\sum_{j=1}^n \mathrm{Var}(\hat y_j) / n = \sigma^2 \normf{X_\cS^\dag X}^2 / n$. Finally, E-optimal design minimizes the largest eigenvalue of the covariance, i.e., the worst-case variance of $\hat\beta$ over all unit directions, which equals $\sigma^2 \norms{X_\cS^\dag}^2$.

We use the superconductivity dataset~\cite{HAMIDIEH2018346} for predicting the critical temperature of superconductors, preprocessed by standardizing the rows and adding an intercept row. The resulting matrix $X$ has shape $82 \times 21236$ and condition number $797.4$; all its rows have Euclidean norm $\sqrt{n}$, and all rows except the intercept row have zero mean. The number of experiments $k$ is swept from $m$ to $8m$. We report the design criteria rather than the prediction error, which is dominated by the large residual of the linear model. Figures~\ref{fig:superconductivity_frob} and~\ref{fig:superconductivity_spec} display results in terms of Frobenius norm objectives and spectral norm objectives, respectively. We use a square-root scale for the $k$-axis, which expands the region $k \approx m$ and makes differences between algorithms visible there.

\begin{figure}[htbp]
\centering
\includegraphics[width=\textwidth]{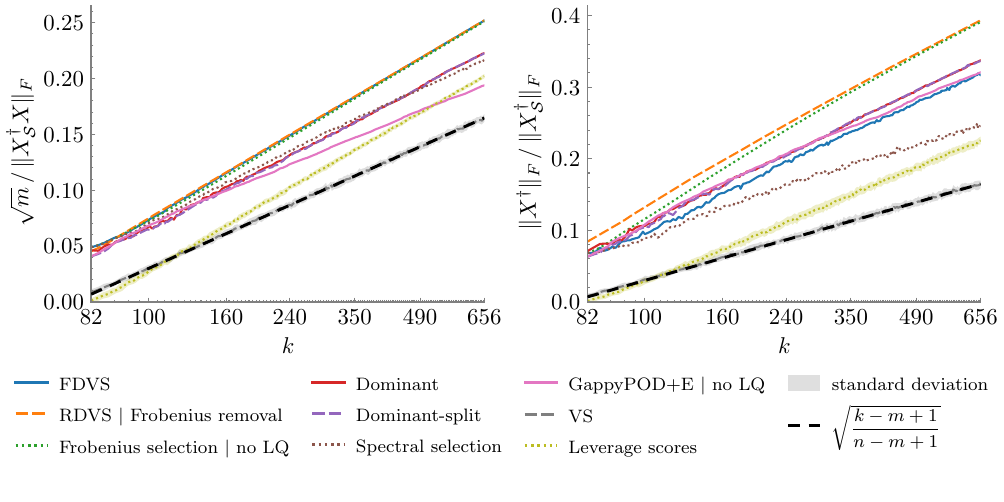}
\caption{Frobenius-norm-based objectives on the superconductivity dataset, $m = 82, n = 21236$: $\sqrt{m}/\normf{X_\cS^\dag X}$ (V-optimality), $\normf{X^\dag}/\normf{X_\cS^\dag}$ (A-optimality), larger is better for both objectives. Notation with \enquote{|} is explained in Items~\ref{itm:vertical_bar} and~\ref{itm:vertical_bar_2}.}
\label{fig:superconductivity_frob}
\end{figure}

In terms of $\normf{X_\cS^\dag X} / \sqrt{m}$, all deterministic algorithms perform comparably, with the three algorithms that explicitly target the Frobenius norm (FDVS, RDVS, Frobenius selection) leading within narrow margins of each other. In terms of $\normf{X_\cS^\dag} / \normf{X^\dag}$, the spread is larger, and the algorithms that target the Frobenius norm while accounting for the nonorthonormality of $X$ (Frobenius removal, Frobenius selection no LQ) score significantly better. Randomized algorithms are two (for large $k$) to ten or more (for $k$ close to $m$) times worse in both metrics. Unless the $\matdec{LQ}$ decomposition of $X$ is precomputed, this renders them impractical: the fastest deterministic methods cost little more than the $\matdec{LQ}$ decomposition itself, and hence little more than the randomized methods.

\begin{figure}[htbp]
\centering
\includegraphics[width=\textwidth]{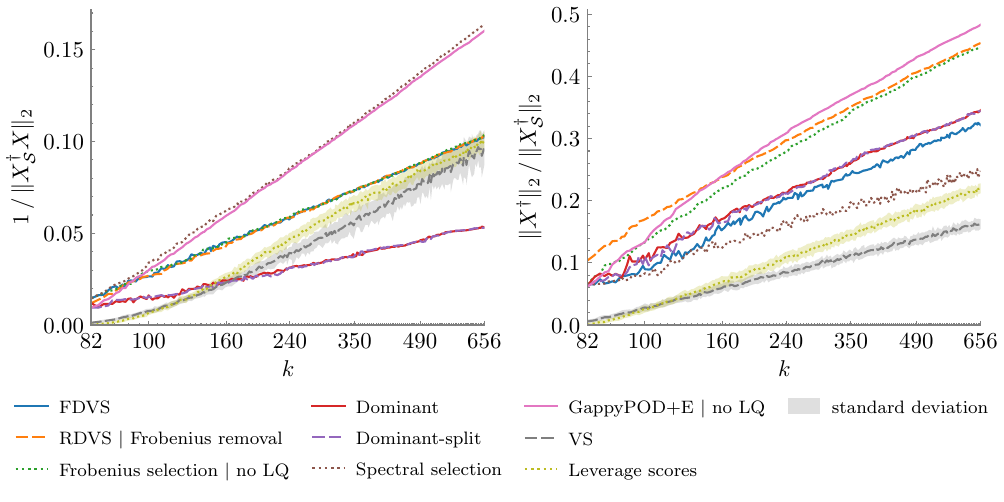}
\caption{Spectral-norm-based objectives on the superconductivity dataset, $m = 82, n = 21236$: $1/\norms{X_\cS^\dag X}$, $\norms{X^\dag}/\norms{X_\cS^\dag}$ (E-optimality), larger is better for both objectives. Notation with \enquote{|} is explained in Items~\ref{itm:vertical_bar} and~\ref{itm:vertical_bar_2}.}
\label{fig:superconductivity_spec}
\end{figure}

Note that while $\normf{X_\cS^\dag X}$ may be similar for different algorithms, the selected submatrices are not, as Figure~\ref{fig:superconductivity_spec} demonstrates. In terms of $\norms{X_\cS^\dag X}$, deterministic algorithms form three distinct groups according to the objective of their greedy procedure: the spectral norm (Spectral selection, GappyPOD+E), the Frobenius norm (FDVS, RDVS, Frobenius selection), and the volume (Dominant, Dominant-split). The first group outperforms the last by a factor of up to almost three, and for large $k$ the volume-based algorithms perform even worse than the randomized ones. In terms of $\norms{X_\cS^\dag} / \norms{X^\dag}$, this pattern is disturbed by the nonorthonormality of $X$, which pushes Frobenius removal and GappyPOD+E no LQ to the top. Randomized algorithms again lag far behind.

\subsection{Sensor placement}

In sensor placement, one reconstructs a high-dimensional state $f \in \R^n$ from measurements taken at only $k$ of its $n$ entries~\cite{sensor_placement}. The state is assumed to be well captured by a low-dimensional basis $U \in \R^{n \times m}$, $U^T U = I$, $m \ll n$, typically the $m$ leading left singular vectors of a matrix of training snapshots. Writing $f = U a + e$ with coefficients $a \in \R^m$ and a truncation residual $e$, the sensors at positions $\cS$ read $f_\cS = U_\cS a + e_\cS + \eta$, where $\eta \sim \mathcal{N}(0, \sigma^2 I)$ is the measurement noise, and the state can be recovered as $\hat f = U U_\cS^\dag f_\cS$. Its error $\hat f - f = U U_\cS^\dag (e_\cS + \eta) - e$ depends on the placement through the amplification factor $\normxi{U U_\cS^\dag} = \normxi{U_\cS^\dag}$, in the mean square for $\xi = F$ and in the worst case for $\xi = 2$. With $X = U^T$, this is precisely Problem~\ref{prm:subset_selection}, and since $X$ has orthonormal rows, $\normxi{X_\cS^\dag X} = \normxi{X_\cS^\dag}$, so both groups of our objectives coincide here.

We reproduce the benchmark of~\cite{sensor_placement} on the NOAA OI SST V2 weekly mean sea surface temperature dataset~\cite{AnImprovedInSituandSatelliteSSTAnalysisforClimate}, where masking out land leaves $n = 44219$ ocean points per weekly field. We take $U$ to be the $m = 302$ leading left singular vectors of the first $1400$ weeks, centered by their temporal mean, and hold out the following $200$ weeks. The sensor count $k$ is swept from $m$ to $8m$.

Each held-out week is reconstructed from the selected sensors, and we report the relative reconstruction error in the Euclidean norm, averaged over the held-out weeks. We plot it against the submatrix metrics to show how well each of them predicts the error. The error of a fixed sensor set varies across weeks more than it differs between algorithms, so the plain standard error would hide consistent differences. We therefore compute the standard error after subtracting, for each week, the mean error across all algorithms; it does not exceed $0.008$ and fits inside the markers. Randomized algorithms are omitted from this figure so that the deterministic ones remain distinguishable.

\begin{figure}[htbp]
\centering
\includegraphics[width=\textwidth]{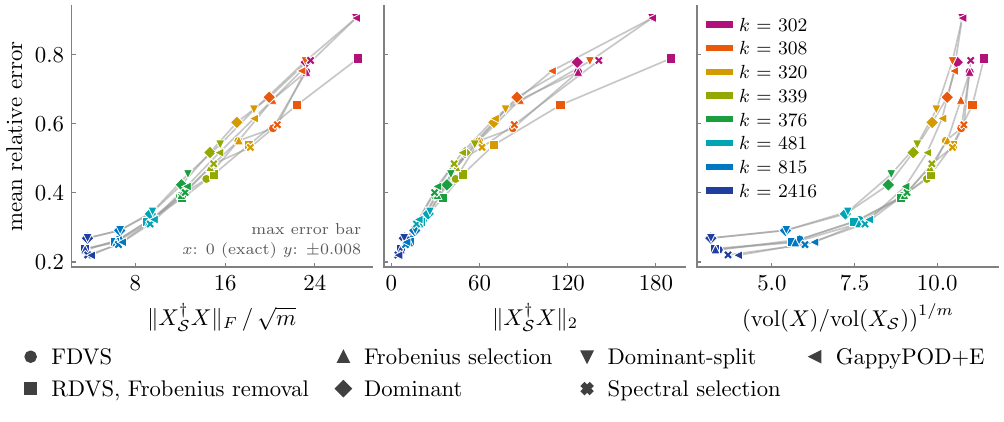}
\caption{Sensor placement on the NOAA SST dataset, $m = 302, n = 44219$: relative reconstruction error against three subset selection metrics. Colors denote the number of sensors $k$, markers the algorithms (Item~\ref{itm:comma} explains the notation with \enquote{,}).}
\label{fig:sensor_placement}
\end{figure}

The results in Figure~\ref{fig:sensor_placement} reinforce the observations from the previous subsection: for each of the three submatrix quality metrics, the algorithms that directly optimize it are the top performers. Notable exceptions are RDVS, which performs poorly in terms of $\normf{X_\cS^\dag X}$ for $k$ close to $m$, and Dominant, which is among the top performers for small $k$ regardless of the metric. While the mean relative error is largely governed by $\normf{X_\cS^\dag X}$, both in theory and in practice, some algorithms consistently yield a lower (RDVS, Frobenius removal) or higher (Dominant, Dominant-split) error than other algorithms that select submatrices with the same $\normf{X_\cS^\dag X}$. Therefore, achieving the best possible performance requires evaluating the algorithms on the target metric of the application itself.

\subsection{DLRA-DEIM}

DLRA-DEIM~\cite{carrel2025} combines dynamical low-rank approximation (DLRA)~\cite{doi:10.1137/050639703}, a framework for approximately solving matrix and tensor differential equations on a low-rank manifold, with the discrete empirical interpolation method (DEIM)~\cite{doi:10.1137/090766498}, which speeds up the computation of projections onto the tangent space of that manifold.

It is known that for many nonlinear matrix or tensor ODEs $\dot{Y}(t) = F(Y(t))$, the solution $Y(t)$ is well captured by a low-rank manifold $\mathcal{M}$ (e.g., the manifold of rank-$m$ matrices, tensors with Tucker rank $m$, or tensor trains with rank $m$) at every time $t$. DLRA aims to reduce the computational cost of solving such ODEs by projecting them onto the manifold: $\dot{Y}(t) = \bP_{T_{Y(t)}\mathcal{M}} F(Y(t))$, where $\bP_{T_{Y(t)}\mathcal{M}}$ is the orthogonal projector onto the tangent space of $\mathcal{M}$ at a point $Y(t)$. Robust numerical integrators for this equation, such as the projector-splitting~\cite{lubich2014projector} and basis-update \& Galerkin~\cite{ceruti2022unconventional} integrators, access $F$ only through its projections onto the tangent space or onto the spaces spanned by the current low-rank factors. When $F$ maps low-rank arguments to low-rank values, these projections are cheap to compute, and DLRA offers a significant speedup~\cite{doi:10.1137/050639703}.

For some equations, however, $Y(t)$ is well captured by a low-rank manifold while $F(Y(t))$ is not. An example is the space-discretized Allen--Cahn equation, where $F$ involves elementwise powers. Computing the exact orthogonal projections is then expensive. DLRA-DEIM addresses this issue with DEIM-style oblique projectors, which require evaluating only a fraction of the entries of the right-hand side. Constructing such projectors in the matrix~\cite{carrel2025}, Tucker tensor~\cite{carrel2025}, and tensor train~\cite{DEKTOR2025143} cases involves selecting pivot indices from matrices of singular vectors, factor matrices $U^{(i)}$, or submatrices of interface matrices, respectively. In each case, the error of the oblique projector relative to the orthogonal one is governed by $\normxi{X_\cS^\dag}$, $\xi \in \{F, 2\}$, where $X$ is the matrix the indices are selected from. In the matrix and Tucker cases, $X$ has orthonormal rows, so $\normxi{X_\cS^\dag}$ and $\normxi{X_\cS^\dag X}$ coincide and the distinction between the two groups of objectives from Items~\ref{itm:vertical_bar} and~\ref{itm:vertical_bar_2} does not arise. In the tensor train case, $X$ does not have orthonormal rows, but it is nearly orthonormal: in our experiments, its condition number never exceeds $1.22$. The two groups of objectives therefore nearly coincide, and the variants that operate on $X$ directly (Frobenius removal and the \enquote{no LQ} variants) produce results almost identical to their orthogonalizing counterparts, so we omit them from the figures.

We solve the three-dimensional Allen--Cahn equation
\[
\frac{\partial y}{\partial t} = \kappa \Delta y + y - y^3
\]
on $\Omega = [0, 2\pi]^3$ with periodic boundary conditions, $\kappa = 0.1$, and the initial condition from~\cite{DEKTOR2025143}. The Laplacian is discretized by second-order central finite differences on a uniform grid with $N = 256$ points per dimension and step $h = 2\pi / N$, which yields a tensor ODE $\dot{Y}(t) = F(Y(t))$ with $Y(t) \in \R^{N \times N \times N}$. We integrate it on $t \in [0, 10]$ in the tensor train format with all ranks equal to $m = 16$, using the tensor train collocation method of~\cite{DEKTOR2025143} and the explicit midpoint (second-order Runge--Kutta) method with time step $\Delta t = h^2 / (12\kappa)$. Snapshots of the solution are shown in Figure~\ref{fig:allen_cahn_cutaway}. In all DLRA-DEIM experiments, we select exactly $k = m$ indices: oversampling ($k > m$) requires additional design choices in how the resulting projector is combined with the rank-$m$ integrator, which are beyond the scope of this paper.

\begin{figure}[htbp]
\centering
\includegraphics[width=\textwidth]{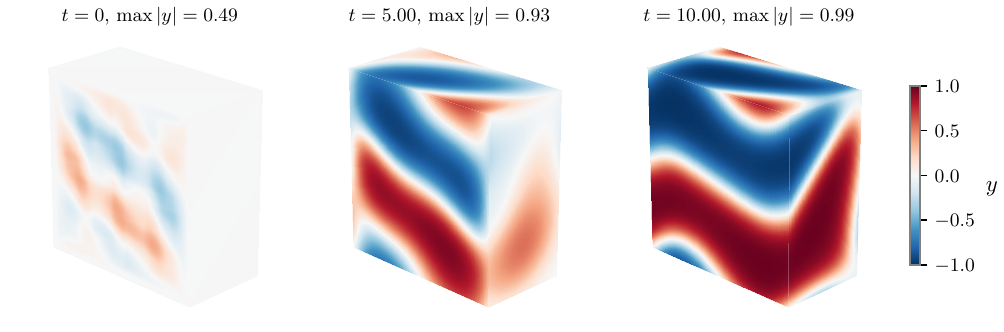}
\caption{Cutaway view of the solution $y$ of the Allen--Cahn equation at times $t = 0$, $t = 5$, $t = 10$.}
\label{fig:allen_cahn_cutaway}
\end{figure}

\begin{figure}[htbp]
\centering
\includegraphics[width=\textwidth]{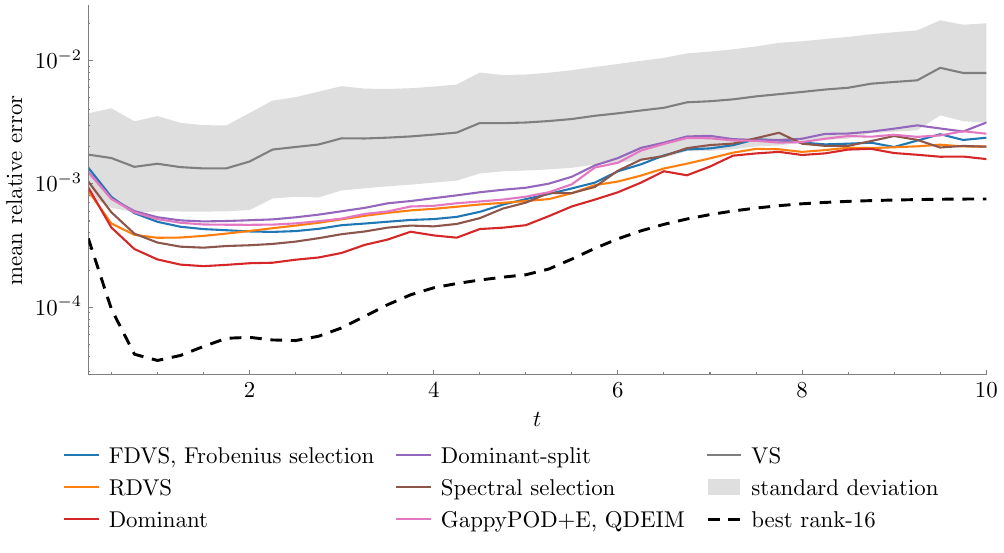}
\caption{Relative error in the Frobenius norm over time in solving the Allen--Cahn equation with DLRA-DEIM, $k = m = 16$. The shaded area shows one standard deviation for VS. Notation with \enquote{,} is explained in Items~\ref{itm:comma_2} and~\ref{itm:comma_3}.}
\label{fig:allen_cahn_trajectories}
\end{figure}

Figure~\ref{fig:allen_cahn_trajectories} shows the relative error $\normf{y - y_{\rm ref}} / \normf{y_{\rm ref}}$ over time for each selection algorithm, where $y_{\rm ref}$ is the reference solution computed by a dense solve on the full grid, together with the error of the best rank-$16$ approximation of $y_{\rm ref}$. All deterministic algorithms outperform VS by a factor of two to five, while the integration with leverage score sampling diverges and is therefore not shown; this again renders randomized selection impractical. Among the deterministic algorithms, Dominant yields the smallest error, roughly half that of most other methods up to $t \approx 6$, after which the errors of all deterministic algorithms become comparable.

\begin{figure}[htbp]
\centering
\includegraphics[width=\textwidth]{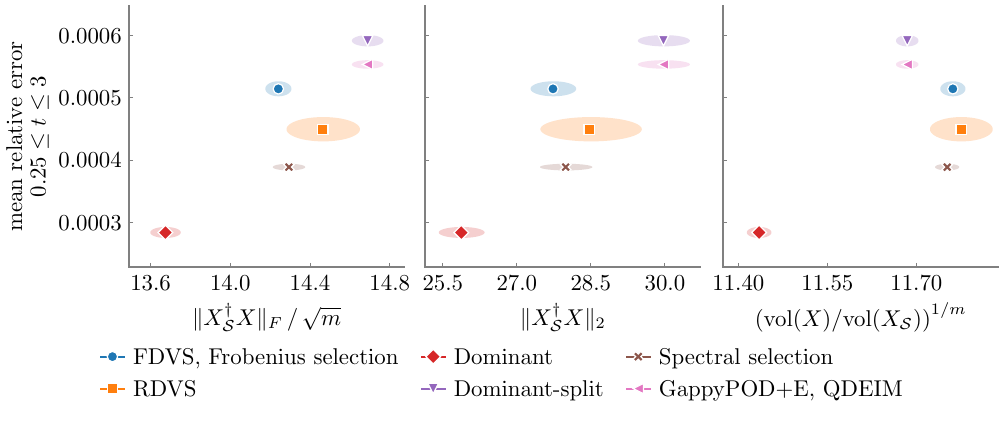}
\caption{Mean relative error in the Frobenius norm on $0.25 \leq t \leq 3$ against three subset selection metrics in solving the Allen--Cahn equation with DLRA-DEIM, $k = m = 16$. Ellipses show the standard deviation over time steps after subtracting the mean across all algorithms at each step. Notation with \enquote{,} is explained in Items~\ref{itm:comma_2} and~\ref{itm:comma_3}.}
\label{fig:allen_cahn_integration_vs_metrics}
\end{figure}

Comparing subset selection algorithms is most meaningful when the right-hand side is well captured by the tangent space, so that the interpolation error is governed by $\normxi{X_\cS^\dag}$. To identify this regime, we measure the relative error $\normf{\Delta F} / \normf{F}$ of the right-hand side cross-interpolated from the selected indices:
\begin{enumerate}
    \item For $0 \leq t \leq 0.25$, this error is large because the state itself is poorly represented in the low-rank format: at $t = 0.125$, it exceeds $100\%$ for some selectors.
    \item For $0.25 \leq t \leq 3$, its geometric mean over all selectors is $1.7\%$.
    \item For $t > 3$, the geometric mean grows to $94\%$, and $48\%$ of the measured values exceed $100\%$. This deterioration occurs because the solution approaches a steady state, where the right-hand side is dominated by noise.
\end{enumerate}
We therefore restrict the comparison to $0.25 \leq t \leq 3$. Figure~\ref{fig:allen_cahn_integration_vs_metrics} relates the mean relative error on this interval to the submatrix quality metrics.

The error generally grows with both $\normf{X_\cS^\dag X}$ and $\norms{X_\cS^\dag X}$ but shows no clear dependence on the volume, in agreement with the fact that the error of the oblique projector is governed by $\normxi{X_\cS^\dag}$, $\xi \in \{F, 2\}$, which nearly coincides with $\normxi{X_\cS^\dag X}$ here. In line with the sensor placement experiment, Dominant attains the best values of both $\normf{X_\cS^\dag X}$ and $\norms{X_\cS^\dag X}$ for $k = m$ and, consequently, the smallest error. Note that the spread in $\normf{X_\cS^\dag X}$ between the algorithms is below $10\%$, yet it translates into a twofold difference in the error.

\section{Conclusion}\label{sec:conclusion}

In this paper, we studied volume sampling and its derandomization as tools for Problem~\ref{prm:subset_selection}. Our main theoretical result is FDVS (Algorithm~\ref{alg:fdvs}), which deterministically reaches the bound $\normf{X_\cS^\dag X}^2 \leq m (n - m + 1)/(k - m + 1)$ in $O(nkm)$ operations. This bound is the tightest known for the problem, and all previous algorithms that provably attain it in polynomial time are at least quadratic in $n$ when $k \ll n$; FDVS is the first to do so with complexity linear in $n$. Together with FIVS (Algorithm~\ref{alg:forward_vol_sample}), RIVS~\cite{derezinski2018reverse}, and RDVS (Algorithm~\ref{alg:lq_frob_rem}), it completes a unified picture in which volume sampling is performed either by forward selection or by reverse removal of columns, in either randomized or derandomized form. 

Within this picture, we identified Algorithm~2 in~\cite{faster_subset_selection} as the reverse derandomized variant and showed how to implement it without $\matdec{SVD}$ downdating, which requires the fast multipole method~\cite{doi:10.1137/S0895479893251472} to run in $O(nm)$ operations and thus carries a large hidden constant that is polylogarithmic in precision. The same machinery yields Frobenius selection (Algorithm~\ref{alg:frob_sel}), an $O(nkm)$ algorithm with a very small constant. Its guarantees in terms of $\normf{X_\cS^\dag X}$ are strong only for $k$ close to $m$, and it is not a form of volume sampling for $k > m$; yet in our experiments, it was the fastest way to obtain a submatrix with near-best $\normf{X_\cS^\dag X}$.

Finally, our numerical experiments in optimal experimental design, sensor placement, and DLRA-DEIM suggest the following practical guidelines.
\begin{enumerate}
    \item The choice of algorithm should be driven by the target metric ($\normf{X_\cS^\dag X}$, $\norms{X_\cS^\dag X}$, $(\vol(X)/\vol(X_\cS))^{1/m}$, etc.), since algorithms whose greedy step targets a given objective tend to perform best in it.
    \item In terms of $\normf{X_\cS^\dag X}$, all tested deterministic algorithms performed comparably, with the relative spread between them not exceeding $20\%$. The practical question is therefore how quickly a good submatrix can be obtained, and here the newly developed Frobenius selection (Algorithm~\ref{alg:frob_sel}) is the clear leader. If tight theoretical bounds for $k > m$ are required, FDVS (Algorithm~\ref{alg:fdvs}) is the method of choice. An exception to this comparability is the regime of $k$ close to $m$, where Dominant~\cite{Osinsky23} can yield noticeably better results, as in the sensor placement and DLRA-DEIM experiments; this, however, is not universal, as the optimal experimental design experiment shows.
    \item The $\norms{X_\cS^\dag X}$ objective is harder: selectors can differ by a factor of almost three, and good results require Spectral selection~\cite{subset_spectral} or the cheaper but slightly worse GappyPOD+E~\cite{doi:10.1137/19M1307391}.
    \item Currently, the only algorithm that fully accounts for the nonorthonormality of $X$ is Frobenius removal (Algorithm~\ref{alg:frob_rem}), which yields better results in terms of $\normf{X_\cS^\dag} / \normf{X^\dag}$ and $\norms{X_\cS^\dag} / \norms{X^\dag}$ on highly nonorthonormal matrices. Whether this justifies its $O(nm^2 + nm(n - k))$ complexity depends on the application.
\end{enumerate}

Overall, volume sampling provides a common framework for designing fast subset selection algorithms with strong guarantees in terms of $\normf{X_\cS^\dag X}$, and in our experiments algorithms built on it were almost always the best in this metric. By making volume sampling and its derandomization asymptotically cheaper, in particular linear in $n$, this work extends these approaches to regimes that were previously out of reach, such as selecting a small number of columns from extremely large matrices arising in optimal experimental design, sensor placement, and model order reduction.

\bibliographystyle{abbrv}
\bibliography{refs}

\appendix

\section{Proof of the volume identity}\label{app:frob_through_vol}

Here we prove identity~\eqref{eqn:frob_through_vol}, used in Section~\ref{sec:intro_vol} to motivate the bound~\eqref{eqn:maxvol_bound}. Throughout, $\cS$ has cardinality $k \in \overline{m, n}$ and $\rank X_\cS = m$.

\textbf{Base case $k = m$.} Without loss of generality, assume $\cS = \overline{1, m}$. Then
\[
\normf{X_\cS^\dag X}^2 = \sum_{j = 1}^n \norms{X_\cS^{-1} x_j}^2 \overset{(a)}{=} \sum_{j = 1}^n\sum_{r = 1}^m \frac{\vol^2(X_{\cS - r + s})}{\vol^2(X_\cS)} \overset{(b)}{=} \sum_{r \in \cS,\ s \notin \cS - r}\frac{\vol^2(X_{\cS - r + s})}{\vol^2(X_\cS)}\,,
\]
where $(a)$ follows from Cramer's rule, and $(b)$ follows from the fact that all volumes of matrices with repeated columns vanish. This is exactly~\eqref{eqn:frob_through_vol} for $k = m$.

\textbf{Induction step.} Let $k > m$ and assume~\eqref{eqn:frob_through_vol} holds for cardinality $k - 1$. Equation~2.2 in~\cite{derezinski2018reverse}, applied to $X_\cS$, yields
\begin{equation}\label{eqn:vol_removal_identity}
\normf{X_\cS^\dag X}^2 = \frac{1}{k - m + 1}\sum_{r \in \cS}\frac{\vol^2(X_{\cS - r})}{\vol^2(X_\cS)}\normf{X_{\cS - r}^\dag X}^2\,.
\end{equation}
The $(k - m + 1)^{-1}$ here is the normalization factor of volume sampling, see~\cite[Equation~2.1]{derezinski2018reverse}. Substituting~\eqref{eqn:frob_through_vol} into the right-hand side of~\eqref{eqn:vol_removal_identity} gives
\begin{multline*}
\normf{X_\cS^\dag X}^2 = \frac{1}{k - m + 1} \sum_{r \in \cS}\frac{\vol^2(X_{\cS - r})}{(k - m)\vol^2(X_{\cS})} \left(\sum_{r' \in \cS - r}\sum_{s \notin \cS - r - r'} \frac{\vol^2(X_{\cS - r - r' + s})}{\vol^2(X_{\cS - r})}\right)\\ - \frac{1}{k - m + 1} \sum_{r \in \cS}\frac{\vol^2(X_{\cS - r})}{(k - m)\vol^2(X_{\cS})}(k - m - 1)(n - k + 2 - m)\,,
\end{multline*}
where we assume that $0 \cdot \infty = 0$. Separating the terms with $s = r'$ and rearranging the sums yields
\begin{multline*}
\normf{X_\cS^\dag X}^2 = \frac{1}{k - m + 1} \sum_{r \in \cS,\ s \notin \cS - r}\frac{\vol^2(X_{\cS - r})}{(k - m)\vol^2(X_{\cS})} \left(\sum_{r' \in \cS - r} \frac{\vol^2(X_{\cS - r - r' + s})}{\vol^2(X_{\cS - r})}\right)\\ - \frac{1}{k - m + 1} \sum_{r \in \cS}\frac{\vol^2(X_{\cS - r})}{(k - m)\vol^2(X_{\cS})}\Bigl((k - m - 1)(n - k + 2 - m) - (k - 1)\Bigr)\,.
\end{multline*}

Finally, applying the generalized Cauchy--Binet formula~\cite[Equation~2.1]{derezinski2018reverse} gives
\[
\sum_{r' \in \cS - r}\vol^2(X_{\cS - r - r' + s}) = (k - m)\vol^2(X_{\cS - r + s}) - \vol^2(X_{\cS - r})\,,
\]
since the $(k - 1)$-element subsets of $\cS - r + s$ are $\cS - r$ itself and the sets $\cS - r - r' + s$ for $r' \in \cS - r$. We obtain~\eqref{eqn:frob_through_vol} for $\cS$ by
\[
(n - k + 1) + (k - m - 1)(n - k + 2 - m) - (k - 1) = (k - m)(n - k + 1 - m)\,,
\]
which completes the induction.

\end{document}